\documentclass[a4paper,11pt,reqno]{amsart}
\usepackage[utf8]{inputenc}
\usepackage{mathtools}
\usepackage{amssymb}
\usepackage{amsfonts}
\usepackage{amsthm}
\usepackage{bbm}
\usepackage{fullpage}
\usepackage{color}
\usepackage{hyperref}
\usepackage{stmaryrd}
\usepackage{enumitem}
\usepackage{accents}

\newcommand{\N}{\mathbb N}
\newcommand{\Z}{\mathbb Z}
\newcommand{\R}{\mathbb R}

\newcommand{\eps}{\varepsilon}
\renewcommand{\P}{\mathbb P}
\newcommand{\E}{\mathbb E}
\newcommand{\1}{\mathbbm 1}
\newcommand{\0}{\mathbf 0}
\newcommand{\F}{\mathcal F}
\newcommand{\G}{\mathcal G}

\newtheorem{theorem}{Theorem}[section]
\newtheorem{lemma}[theorem]{Lemma}

\newtheorem{corollary}[theorem]{Corollary}

\newtheorem{thmx}{Theorem}

\newtheorem{remark}[theorem]{Remark}

\theoremstyle{definition}

\newcommand{\isDistr}{\overset d=}
\newcommand{\dd}{\text{d}}

\newcommand{\cev}[1]{\accentset{\leftarrow}{#1}}

\newcommand{\p}{{p^*}}
\newcommand{\f}{\mathfrak{f}}

\newcounter{constants}

\newcommand{\PSRW}{P^{\text{SRW}}}

\newcommand{\llambda}{{\boldsymbol{\lambda}}}
\newcommand{\EE}{\mathcal E}
\newcommand{\M}{\mathcal M}

\newcommand{\m}{\boldsymbol{m}}
\newcommand{\pp}{p^*_{\operatorname{pinn.}}}
\renewcommand{\t}{\boldsymbol{t}}

\title{Local limit theorem for directed polymers beyond the $L^2$-phase}
\author{Stefan Junk}
\address{Gakushuin University, 1--5--1 Mejiro, Toshima-ku, Tokyo 171-8588 Japan}
\email{sjunk@math.gakushuin.ac.jp}
\date{\today}

\begin{document}
\lineskip=0pt

\begin{abstract}
We consider the directed polymer model in the weak disorder phase under the assumption that the partition function is $L^p$-bounded for some $p>1+\frac{2}d$. We prove that the point-to-point partition function can be approximated by two point-to-plane partition functions at the startpoint and endpoint, and in particular that it is $L^p$-bounded as well. Some consequences of this result are also discussed, the most important of which is a local limit theorem for the polymer measure. We furthermore show that the required $L^p$-boundedness holds for some range of $\beta$ beyond the $L^2$-critical point, and in the whole interior of the weak disorder phase for environments with finite support.
\end{abstract}

\maketitle

\section{Introduction}

\subsection{Overview}\label{sec:intro}
The directed polymer model describes random paths in a disordered medium. The model has recently attracted much interest because it is conjectured to be in the KPZ (Kardar-Parisi-Zhang) universality class. In the so-called \emph{strong disorder} regime, in particular in spatial dimension $d=1$, it is expected that the polymer has a super-diffusive scaling exponent and that its behavior is thus completely different from its infinite-temperature version (the usual simple random walk). At present, this has only been verified in a small number of exactly solvable models.

\smallskip In contrast to the one-dimension case, in spatial dimensions $d\geq 3$ it is known that the diffusive scaling of the simple random walk persists up to some inverse temperature $\beta_{cr}>0$. This parameter regime is known as the \emph{weak disorder} phase and it is the focus of the current article. It is characterized as the set of $\beta$ such that the (normalized) partition function $W_n^\beta$ converges to a positive limit $W_\infty^\beta$. The long-term behavior in the weak disorder phase is much better understood than in the strong disorder phase, see for example \cite{B89, AS96, CY06}, but many important questions still remain.

\smallskip We are particularly interested in the case where $\beta$ is close to $\beta_{cr}$, which is an intriguing regime because strong disorder holds beyond $\beta_{cr}$ (it has recently been proved that $\beta_{cr}$ itself belongs to the weak disorder phase \cite{JL24,JL25_2}). However, this regime is technically difficult because a successful approach to the weak disorder phase, going back to \cite{B89} and based on $L^2$-martingale techniques, is not applicable in the full weak disorder regime but only up to some $\beta_{cr}^{L^2}$, which is known to be strictly smaller than $\beta_{cr}$. The contribution of this paper is to introduce an approach based on $L^p$-estimates that is valid beyond $\beta_{cr}^{L^2}$, and up to $\beta_{cr}$ for a certain class of environments.

\smallskip Our main result is a moment bound for the point-to-point partition $W_{0,n}^{\beta,0,x}$ between space-time points $(0,0)$ and $(n,x)$, uniformly in $n\in\N$ and in a large range of $x$, which in particular implies a local limit theorem for the polymer measure. Informally, the latter result says that the density of the polymer measure $\mu_{\omega,n}^\beta$ (the quenched law of the polymer up to time $n$) is comparable to the density of the simple random walk with a random multiplicative constant that is well-behaved. More precisely, one can write
\begin{align*}
\mu_{\omega,n}^\beta(X_n=x)=\frac{W^{\beta,0,x}_{0,n}}{W^\beta_n}\PSRW(X_n=x),
\end{align*}
and our main contribution is to show that the numerator is bounded away from zero and infinity and factorizes as $W_{0,n}^{\beta,0,x}\approx W_\infty^\beta \cev W_{-\infty}^{\beta,n,x}$ where $\cev W^{\beta,n,x}_{-\infty}$ denotes the ``backward'' partition function started from space-time point $(n,x)$. Since the denominator converges to $W_\infty^\beta>0$ in weak disorder, we obtain a local limit theorem of the form
\begin{align}\label{eq:informal}
\mu_{\omega,n}^\beta(X_n=x)\approx \cev W_{-\infty}^{\beta,n,x}\PSRW(X_n=x).
\end{align}
See Corollary~\ref{cor:informal} for a precise statement.

\smallskip The most difficult step in this outline is the moment bound for $W_{0,n}^{\beta,0,x}$, which proves that the contribution from the environment in the bulk (i.e., from sites $(t,x)$ with $0\ll t\ll n$) is negligible and justifies the factorization of the numerator used in \eqref{eq:informal}. Note that the integrability of $W_n^\beta$ in weak disorder is by now quite well-understood, but moment bounds for $W_{0,n}^{\beta,0,x}$ are conceptually more difficult. In particular, it is a priori not clear that $(W_{0,n}^{\beta,0,0})_{n\in 2\N}$ is uniformly integrable in weak disorder. The reason is that it is not a martingale, unlike $(W_n^\beta)_{n\in\N}$, and one expects to loose some integrability from fixing the polymer endpoint. Our main contribution is to quantify that loss using a novel form of ``chaos decomposition'', see Section~\ref{sec:strat}.

\smallskip Similar results are known in the $L^2$-weak disorder phase $\beta<\beta_{cr}^{L^2}$, see \cite{S95,V06, CCN22}, and the extension to the full weak disorder phase has been an important open problem in the field. For example, in \cite[Remark 5.1]{C17} it is noted that ``the validity of local limit theorem is a natural definition for the polymer to be in the weak disorder regime (better than the central limit theorem itself)''. We also note that, in spatial dimension $d=2$, strong disorder holds at all $\beta>0$ but one can define an \emph{intermediate weak disorder} regime by choosing a time-dependent inverse temperature $\beta_n$ that goes to zero at a suitable rate. In that case, a local limit theorem is known in the full weak disorder regime, see \cite{G22, NN21}. This is due to the fact that, in contrast to the higher-dimensional regime considered here, $L^2$-techniques are applicable in the whole intermediate weak disorder regime.

\smallskip Limit theorems of similar forms have been obtained in the literature on random media many times in different contexts, see for example \cite{BCR16, SS04, BB07, BBDS23, BD10}, and have proven a valuable tool in situations where, in some sense, the effect of disorder is weak and the model homogenizes. In our setup, \eqref{eq:informal} allows a more precise comparison between the simple random walk and the polymer measure than what is provided by, for example, the invariance principle. The reason is that  to compute finer properties of $\mu_{\omega,n}^\beta(X_n\in\cdot)$ it is enough to understand the quality of the approximation in \eqref{eq:informal} and the statistics of $\cev W^{\beta,n,x}_\infty$ over an appropriate range of $x$. We demonstrate this approach in Section~\ref{sec:appl} by deriving a number of new results for the weak disorder phase.

\smallskip Let us finally comment on the range of validity of our results. We prove the moment bound for the point-to-point partition function (Theorem~\ref{thm:local}) and its consequences, including the precise version of the local limit theorem \eqref{eq:informal} in Corollary~\ref{cor:informal}, under the assumption that $(W_n^\beta)_{n\in\N}$ is $L^p$-bounded, for some $p>1+2/d$. We further show in Corollary~\ref{cor:nontriv} that, for general environments satisfying only the usual assumption \eqref{eq:exp_mom}, this condition is valid in an interval $[0,\beta_{cr}^{L^{1+2/d}})$ that is strictly larger than $[0,\beta_{cr}^{L^2})$. In particular, our result extends previously known versions of the local limit theorem in the $L^2$-phase. We strongly believe that $\beta_{cr}^{L^{1+2/d}}$ equals $\beta_{cr}$, i.e., that Theorem~\ref{thm:local} and its consequences are valid throughout the interior of the weak disorder phase, but we can only confirm this under a technical assumption (see Corollary~\ref{cor:fin}). We hope to remove this restriction in future work and conjecture that $\beta_{cr}^{L^{1+2/d}}=\beta_{cr}$ for general environments.

\subsection{Definition of the model}\label{sec:def}

We now introduce the model in detail. The random medium (also called \emph{disorder} or  \emph{random environment}) is given by an independent and identically distributed (i.i.d.) collection of real-valued weights $(\omega=(\omega_{t,x})_{(t,x)\in\N\times \Z^d},\P,\F)$ with finite exponential moments,
\begin{align}\label{eq:exp_mom}
	\E\big[e^{\beta|\omega_{0,0}|}\big]<\infty\qquad\text{ for all }\beta\geq 0.
\end{align}
Assumption \eqref{eq:exp_mom} will be in place throughout the paper. The weight of a path $\pi$ in $\omega$ is given by
\begin{align}\label{eq:HI}
	H_I(\omega,\pi)\coloneqq \sum_{i\in I\cap \N}\omega_{i,\pi(i)},
\end{align}
with $H_n\coloneqq H_{[1,n]}$, and the polymer measure $\mu_{\omega,n}^\beta$ is defined to be the associated Gibbs-measure,
\begin{align}\label{eq:mu}
	\mu_{\omega,n}^\beta(\dd X)\coloneqq\frac{1}{W_n^\beta}e^{\beta H_n(\omega,X)-n\lambda(\beta)}P(\dd X),
\end{align}
where $\lambda(\beta)\coloneqq \log \E[e^{\beta\omega_{0,0}}]$ is the logarithmic moment generating function, $(X=(X_n)_{n\in\N},P)$ is the simple random walk on $\Z^d$ and $W_n^\beta$ is the normalizing constant,
\begin{align*}
	W_n^\beta=E[e^{\beta H_n(\omega,X)-n\lambda(\beta)}].
\end{align*}
We will refer to this quantity as the partition function. It is not difficult to check that $(W_n^\beta)_{n\in\N}$ is a non-negative martingale and hence the almost sure limit $W_\infty^\beta\coloneqq \lim_{n\to\infty}W_n^\beta$ exists. One can further show that a zero-one law holds, i.e. $\P(W_\infty^\beta>0)\in\{0,1\}$. The formal definition of the weak disorder regime mentioned above is as the set of $\beta\geq 0$ where
\begin{align}
	\P(W_\infty^\beta>0)=1\tag{WD}\label{eq:WD}.
\end{align}
The strong disorder phase is the complementary set of $\beta$, satisfying
\begin{align}\label{eq:SD}\tag{SD}
\P(W_\infty^\beta=0)=1.
\end{align}

We now highlight some important results and refer to the survey articles \cite{C17, CSY04, D09, Z24} for further details. The next theorem shows that there is a phase transition between weak and strong disorder.

\begin{thmx}\label{thmx:A}
	\begin{enumerate}
		\item[(i)]  There exists $\beta_{cr}=\beta_{cr}(d)\in[0,\infty]$ such that \eqref{eq:WD} holds for $\beta<\beta_{cr}$ and \eqref{eq:SD} holds for $\beta>\beta_{cr}$. If $\beta_{cr}<\infty$, then \eqref{eq:WD} also holds at $\beta_{cr}$.
		\item[(ii)] We have $\beta_{cr}(d)=0$ for $d<3$ and $\beta_{cr}(d) \in(0,\infty]$ for $d\geq 3$. Moreover, for $d\geq 3$ there exists $\beta_{cr}^{L^2}=\beta_{cr}^{L^2}(d)\ \in (0,\infty]$ such that $(W_n^\beta)_{n\in\N}$ is $L^2$-bounded if and only if $\beta<\beta_{cr}^{L^2}$.
		\item [(iii)] For $d\geq 3$, if $\beta_{cr}^{L^2}<\infty$ then $\beta_{cr}^{L^2}<\beta_{cr}$.
	\end{enumerate}
\end{thmx}
\begin{proof}
The existence of the phase transition is shown in \cite{CSY03}. The martingale approach and the $L^2$-phase was introduced in \cite{B89}. The fact that $\beta_{cr}^{L^2}$ is strictly smaller than  $\beta_{cr}$ has been shown over a number of works, starting with \cite{B04}, see \cite[Remark 5.2]{C17} for the precise references. If the environment is upper bounded,
		\begin{align}\label{eq:bd}\tag{u-bd.}
		\exists K>0\text{ s.t. }\P(\omega_{0,0}\leq K)=0,
		\end{align}
		then the fact that \eqref{eq:WD} holds at $\beta_{cr}$ has been proved in \cite{JL24}, and the assumption \eqref{eq:bd} was later removed in \cite{JL25_2}.
\end{proof}
Next, we recall some information about the behavior of $\mu_{\omega,n}^\beta$ within each phase. An important observable is the so-called replica overlap
\begin{align}\label{eq:def_I}
	I_n^{\beta,2}\coloneqq \sum_{x\in\Z^d}\mu_{\omega,n}^\beta(X_{n+1}=x)^2=\mu^{\otimes 2, \beta}_{\omega,n}(X_{n+1}=X_{n+1}'),
\end{align}
where $\mu^{\otimes 2,\beta}_{\omega,n}=\mu_{\omega,n}^\beta\otimes\mu_{\omega,n}^\beta$ denotes the law of two independent polymers $(X_n)_{n\in\N}$ and $(X_n')_{n\in\N}$ in the same environment $\omega$.

\begin{thmx}\label{thmx:B}
\begin{enumerate}
	\item[(i)] If \eqref{eq:WD} holds, then $\sum_nI_n^{\beta,2}$ is almost surely finite and, in particular, $\lim_{n\to\infty}\max_x\mu_{\omega,n}(X_{n+1}=x)=0$  almost surely. Moreover, for every $f\colon\R^d\to\R$ bounded and continuous
		\begin{align}\label{eq:CLT}
			\sum_{x\in\Z^d}f(x/\sqrt n)\mu_{\omega,n}^\beta(X_n=x)\xrightarrow{n\to\infty}\int f(x)k(x)\dd x\qquad\text{ in probability},
		\end{align}
	 where $k$ denotes the normal density with expectation zero and covariance matrix $\frac 1dI$.
 \item[(ii)] If \eqref{eq:SD} holds, then there exists $c>0$ such that $\liminf_{n\to\infty}\max_{x\in\Z^d}\mu_{\omega,n}^\beta(X_{n+1}=x)\geq c$ almost surely.
\end{enumerate}
\end{thmx}
\begin{proof}
	The central limit theorem \eqref{eq:CLT} is first proved in the $L^2$-phase \cite[Theorem~1.2]{AS96} and later extended to the whole weak disorder phase \cite[Theorem~1.2]{CY06}. The characterization of the phase transition in terms of the finiteness of $\sum_n I_n^{\beta,2}$ is proved in \cite[Theorem~2.1]{CSY03} and part (ii) is proved in \cite[Theorem~1.4.2]{Y10} for a setting that generalizes our model.
\end{proof}

As mentioned earlier, our main result relies on $L^p$-boundedness instead of $L^2$-boundedness, and we thus introduce the critical exponent at  a given inverse temperature $\beta$,
\begin{align}\label{eq:def_p}
	\p(\beta)\coloneqq \sup\big\{p\colon (W_n^\beta)_{n\in\N}\text{ is $L^p$-bounded}\big\}.
\end{align}
We will sometimes write $\p$ instead of $\p(\beta)$ to simplify the notation. A priori, it is not clear that \eqref{eq:WD} guarantees $\p(\beta)>1$ and an argument is needed to ensure this is the case. Somewhat surprisingly, it turns out that $\p(\beta)$ is even bounded away from $1$ throughout the weak disorder phase.
\begin{thmx}\label{thmx:p}
	\begin{enumerate}
		\item[(i)]In \eqref{eq:WD}, it holds that $\E[\sup_n W_n^\beta]<\infty$ and, in particular, $(W_n^\beta)_{n\in\N}$ is uniformly integrable.
		\item [(ii)] In \eqref{eq:WD}, it holds that $\p(\beta)\geq 1+\frac{2}{d}$.
		\item [(iii)] If \eqref{eq:WD} holds at $\beta$, then $\beta\mapsto\p(\beta)$ is left-continuous at $\beta$.
		\item [(iv)] In $d\geq 3$, if $\beta_{cr}^{L^2}<\infty$ then $\p(\beta_{cr}^{L^2})=2$.
	\end{enumerate}
\end{thmx}

\begin{proof}
	\textbf{Part (i)}: The uniform integrability was shown in \cite[Proposition~3.1]{CY06} and the integrability of $\sup_nW_m^\beta$  is proved in \cite[Theorem~1.1]{J21_1}.

	\smallskip \textbf{Part (ii)}: The fact that $\p>1$ was first proved under the assumption \eqref{eq:bd} in \cite[Theorem~1.1]{J21_1}, which was then extended a large class of unbounded environments in \cite{FJ23}. The lower bound $\p\geq 1+\frac 2d$ was proved in \cite[Corollary~1.3]{J25} under the assumption \eqref{eq:bd}, which was later removed in \cite[Corollary 2.7(ii)]{JL24_2}.

	\smallskip \textbf{Part (iii)}: The left-continuity is proved in \cite[Corollary 2.9(i)]{JL24_2}.

	\smallskip \textbf{Part (iv)} follows by a comparison to the so-called inhomogeneous pinning model, see the discussion in \cite[Section~1.4]{BS10} and \cite{BT10} for the case $d=3$. Alternatively, one can note that by Theorem~\ref{thmx:B}(iii), part (ii) and the definitions of $\p$ and $\beta_{cr}^{L^2}$, we have $1<\p(\beta_{cr}^{L^2})\leq 2$ and $\p(\beta)\geq 2$ for $\beta<\beta_{cr}^{L^2}$, hence $\p(\beta_{cr}^{L^2})=2$ follows by left-continuity.
\end{proof}

\subsection{Statement of the key moment bound}\label{sec:local}

The core of this paper is a moment bound for the point-to-point partition function, which is presented in this section. While this result is technical, the bound is quite powerful and has many useful consequences, including the local limit theorem for the polymer endpoint (see Section~\ref{sec:alt}).

\smallskip To state it, we introduce some notation. First, for $m\leq n$ and $x,y\in\Z^d$, we write $(m,x)\leftrightarrow(n,y)$ if $P(X_{n-m}=x-y)>0$. If $(m,x)\leftrightarrow(n,y)$, then the random walk bridge between $(m,x)$ and $(n,y)$ is denoted by
\begin{align*}
P^{x,y}_{m,n}(\cdot)=P(X\in\cdot|X_m=x,X_n=y),
\end{align*}
with expectation $E^{x,y}_{m,n}$, and the corresponding pinned partition function is defined as
\begin{align*}
	W^{\beta,x,y}_{m,n}\coloneqq E^{x,y}_{m,n}[e^{\beta H_{(m,n)}(\omega,X)-(n-m-1)\lambda(\beta)}].
\end{align*}
Note that the environment at $(n,x)$ and $(m,x)$ is ignored. In the calculations, we have to be careful about whether the environment at the initial time and terminal times is included, but we stress that due to \eqref{eq:exp_mom} this difference essentially does not matter. If $\nu$ is a probability measure on $\Z^{2d}$ such that
\begin{align*}
	\nu(\{(x,y)\colon (m,x)\leftrightarrow(n,y)\})=1,
\end{align*}
then we define the averaged pinned partition function by
\begin{align*}
	W^{\beta,\nu}_{m,n}\coloneqq \sum_{x,y}\nu(x,y) W^{\beta,x,y}_{m,n}.
\end{align*}
We also introduce the random walk $P^{x,\star}_{m,n}$ (resp. $P^{\star,y}_{m,n}$) with starting point $(m,x)$ (endpoint $(n,y)$) and free endpoint (starting point). Namely, $P^{x,\star}_{m,n}$ is the law of $(X_k-X_m+x)_{k=m,\dots,n}$ and $P^{\star,y}_{m,n}$ is the law of $(X_k-X_n+y)_{k=m,\dots,n}$ under $P$. The corresponding partition functions are denoted by
\begin{align}
	W^{\beta,x,\star}_{m,n}&\coloneqq E^{x,\star}_{m,n}[e^{\beta H_{(m,n]}(\omega,X)-(n-m)\lambda(\beta)}],\label{eq:forward}\\
	W^{\beta,\star,y}_{m,n}&\coloneqq E^{\star,y}_{m,n}[e^{\beta H_{[m,n)}(\omega,X)-(n-m)\lambda(\beta)}].\label{eq:backward}
\end{align}
In \eqref{eq:backward}, the environment at time $n$ is ignored, which together with the reversibility of the simple random walk ensures that \eqref{eq:forward} and \eqref{eq:backward} have the same law.

\smallskip Finally, we introduce a set of admissible distributions for two endpoints that appear in Theorem~\ref{thm:local}, which are separated by a distance $O(n)$ in space and we allow fluctuations of order $O(n^{1/2})$ around the two endpoints. Let
\begin{equation}\label{eq:def_EE}
	\begin{split}
	&\EE_n(\alpha,M)\\&\coloneqq \left\{\nu\in\M_1(\Z^{2d})\colon \begin{matrix}\exists a,b\in\Z^d\colon |a-b|\leq \alpha n, \nu\big((a,b)+[Mn^{1/2},Mn^{1/2}]^{2d}\big)=1,\\[1mm]\nu(\{(x,y)\colon (0,x)\leftrightarrow(n,y)\})=1\end{matrix}\right\},
	\end{split}
\end{equation}
where $\mathcal M_1(\Z^{2d})$ denotes the set of probability measures on $\Z^{2d}$. We now state the main result.
\begin{theorem}\label{thm:local}
Assume $d\geq 3$ and $\p(\beta)>1+2/d$ and let $p\in(1+2/d,\p\wedge 2)$. There exists $\alpha>0$ such that the following hold:
\begin{enumerate}
	\item[(i)] It holds that
\begin{align}
\sup_{n\in\N}\sup_{|x-y|\leq \alpha n\colon (0,x)\leftrightarrow(n,y)}\E\left[\big(W^{\beta,x,y}_{0,n}\big)^p\right]&<\infty\label{eq:local3}.
\end{align}
\item[(ii)]Let
\begin{align}\label{eq:xi}
	\xi=\xi(\p,p)=\begin{cases}\frac{\p-1}{\p+1}(1-\frac{1+2/d}{\p})&\text{ if }p\in(1+\frac2d,1+\frac{\p-1}{\p+1}),\\
	\frac{\p-p}{\p}(1-\frac{1+2/d}{\p})&\text{ if }p\in(1+\frac{\p-1}{\p+1},\p).\end{cases}
\end{align}

For every $\eps\in(0,\xi)$ and $M>0$ there exists $C>0$ such that
\begin{align}
	\sup_{n\in\N}\sup_{\nu\in\EE_n(\alpha,M)}\big\|W^{\beta,\nu}_{0,n}-1\big\|_p&\leq C\Big(\max_x \nu(x,\star)^{\xi-\eps}+\max_y\nu(\star,y)^{\xi-\eps}\Big),\label{eq:local1}
\end{align}
where $\nu(x,\star)=\sum_y\nu(x,y)$ and $\nu(\star,y)=\sum_x\nu(x,y)$ denote the marginals of $\nu$.
\end{enumerate}
\end{theorem}

To understand the significance of part (ii), recall from the discussion in Section~\ref{sec:intro} that $W^{\beta,x,y}_{0,n}$ depends mostly on the environment close to $(0,x)$ and $(n,y)$. We therefore expect $W_{0,n}^{\beta,x,y}$ and $W_{0,n}^{\beta,x',y'}$ to become independent if $|x-x'|$ and $|y-y'|$ are both large, and thus $W^{\beta,\nu}_{0,n}$ should concentrate around its expectation if $\nu$ is sufficiently spread out. The bound \eqref{eq:local1} makes this precise and gives a quantitative estimate on the decay in term of the largest weight of $\nu$.

\smallskip Note that some restriction on the endpoints $x$ and $y$ in \eqref{eq:local3} is necessary, since, for example, $W^{\beta,0,ne_1}_{(0,n)}=e^{\sum_{i=1}^{n-1}(\beta \omega_{i,ie_1}-\lambda(\beta))}$ is not uniformly integrable, where $e_1$ is the first coordinate vector.

\subsection{Strategy of the proof of Theorem~\ref{thm:local}}\label{sec:strat}

The proof can be considered a variation of the so-called ``chaos decomposition'' that has been very successful in the analysis of the $L^2$-phase, in particular in the proof of a local limit theorem, see \cite{S95, V06, NN21, G22}.

\smallskip To understand the strategy, it is helpful to first recall the strategy for the proof of a moment bound corresponding to Theorem~\ref{thm:local} in \cite{S95}. The idea is to decompose $W^{\beta,x,y}_{0,n}$ as $\sum_k A_k$, where $A_k$ is itself a  sum of many pairwise orthogonal terms. Thus the second moment of $\sum_k A_k$ consists only of the on-diagonal terms, and the contribution from these terms can be interpreted as a joint partition function $W^{\otimes 2,\beta,x,y}_{0,n}[\1_{X_s=X_s'\text{ at least $k$ times}}]$ of two independent polymers in the same environment that are forced to meet at least $k$ times. Here, we have introduced the notation
\begin{align}\label{eq:def_joint_restricted}
	W^{\otimes 2,\beta,x,y}_{0,n}[\1_B]\coloneqq E^{\otimes 2,x,y}_{0,n}[e^{\beta H_{(0,n)}(\omega,X_n)+\beta H_{(0,n)}(\omega,X_n')-2(n-1)\lambda(\beta)}\1_B],
\end{align}
where $P^{\otimes 2,x,y}_{0,n}$ is the law of two independent random walk bridges $(X_k)_{k=0,\dots n}$ and $(X_k')_{k=0,\dots,n}$ and $B$ is an event adapted to their common sigma field.

\smallskip To prove a moment bound for $W_{0,n}^{\beta,x,y}$ in $L^2$-weak disorder, one can verify by explicit calculations that the second moment of $A_k$ decays exponentially fast, hence it suffices to consider the first $O(\log n)$ terms. One can further verify that for $k=O(\log n)$, $A_k$ is dominated by the contribution from paths where all collisions between $X$ and $X'$ occur in $[0,n^{o(1)}]\cup[n-n^{o(1)},n]$. One can then integrate out the environment ``in the middle'' and show that the random walk bridges in the ``initial'' and ``terminal'' parts behave like simple random walks.

\smallskip This procedure does not work outside of the $L^2$-region because the second moment of $A_k$ grows exponentially and we instead work with the $p^{\operatorname{th}}$ moment. By writing
\[
	(W^{\beta,x,y}_{0,n})^p=(W^{\otimes 2,\beta,x,y}_{0,n})^{p/2}
\]
we obtain a joint partition function, which we similarly decompose as $W^{\otimes 2,\beta,x,y}_{0,n}=\sum_k \widehat A_k$. The difference is that we group successive collisions between the two polymers together if they occur in a short time interval, so that the subscript $k$ of $\widehat A_k$ counts the number of meetings that are ``well-separated'' in time. Intuitively, this gives the polymers time to move away from each other after a collision, which we know to be their typical behavior due to Theorem~\ref{thmx:B}(i). With this modification, we can again argue that it suffices to consider the first $O(\log n)$ terms.

\smallskip A second difference to $L^2$-weak disorder is that not all collisions occur in the initial and terminal segments. Instead, we show that there exists a constant $L$ such that there are at most $L$ large gaps between the collision times, where ``large'' means ``of order $n$''. We can then integrate out the environment in a large gap and follow the argument from above.

\smallskip Our decomposition of the $p^{\operatorname{th}}$-moment of the partition function is inspired by the approach in \cite[p. 108]{CC09}, and similar to that work we will often use the following sub-additive estimate,
\begin{align}\label{eq:sa}\tag{sub-add.}
\Big(\sum_{i\in I}x_i\Big)^\theta\leq \sum_{i\in I}x_i^\theta\qquad\text{ for all }\theta\in[0,1]\text{ and all non-negative }(x_i)_{i\in I}.
\end{align}

\subsection{Discussion of the integrability assumption}\label{sec:ass}

Theorem~\ref{thm:local} requires $L^{1+2/d+\eps}$-boundedness of the partition function, so in this section we discuss whether this assumption is actually satisfied for any $\beta>\beta_{cr}^{L^2}$ in the weak disorder phase.

\smallskip First, we recall that due to Theorem~\ref{thmx:p}, we have $\p(\beta)\geq 1+2/d$ in weak disorder. Nonetheless, it could be the case that $\beta\mapsto\p(\beta)$ has a discontinuity at $\beta_{cr}^{L^2}$ and that $\p$ equals $1+2/d$ in $(\beta_{cr}^{L^2},\beta_{cr}]$. To exclude that possibility, we prove the following properties of $\p$.
\begin{theorem}\label{thm:p}
The  function $\beta\mapsto\p(\beta)$ satisfies the following:
\begin{enumerate}
	\item[(i)] If $\p(\beta)\in(1+2/d,2]$, then $\p$ is right-continuous at $\beta$.
	\item[(ii)] Assume that $\P$ has finite support, i.e., there exists $A\subseteq\R$ with $|A|<\infty$ such that $\P(\omega_{0,0}\in A)=1$. If $\p(\beta)>1$, then $\p(\beta')>\p(\beta)$ for all $\beta'<\beta$.
\end{enumerate}
\end{theorem}

\smallskip A first interesting consequence of Theorem~\ref{thm:p} is that it allows to resolve the behavior at $\beta_{cr}$, as has also been noted in \cite[Corollary 2.2]{JL24} and \cite[Corollary~2.5]{JL25_2}.

\begin{corollary}\label{cor:betac}
If $d\geq 3$ and $\beta_{cr}<\infty$, then $\p(\beta_{cr})=1+\frac 2d$.
\end{corollary}
\begin{proof}
Due to Theorems~\ref{thmx:A} and~\ref{thmx:p}, we know that \eqref{eq:WD} holds at $\beta_{cr}$ and that $\p(\beta_{cr})\in[1+\frac 2d,2]$. If $\p(\beta_{cr})>1+\frac 2d$, then $\beta\mapsto\p(\beta)$ is right-continuous at $\beta_{cr}$ by Theorem~\ref{thm:p}(i), hence $\p(\beta)>1$ for some $\beta>\beta_{cr}$. This contradicts the definition of $\beta_{cr}$.
\end{proof}

Theorem~\ref{thm:local} does not apply at $\beta_{cr}$ -- the methods in this paper require that there exists $p$ such that $\sum_{t\in\N,x\in\Z^d}P(X_t=x)^p<\infty$ and $\sup_{n\in\N}\E[(W_n^{\beta})^p]<\infty$, and in the case of $\beta=\beta_{cr}$ and $p=\p(\beta_{cr})=1+\frac 2d$ both of these assertions fail. It remains an open problem to determine whether some form of Theorem~\ref{thm:local} is valid at $\beta_{cr}$. Such a result would presumably shed much light on the behavior of the polymer model at criticality and we expect that it is a prerequisite to a deeper analysis of the phase transition, for example to understand the analyticity of the free energy $\f(\beta)=\lim_{n\to\infty}\frac 1n\E[\log W_n^\beta]$ at $\beta_{cr}$.

\smallskip On the other hand, we can show that Theorem~\ref{thm:local} does apply in a non-trivial range $[\beta_{cr}^{L^2},\beta_{cr}^{L^{1+2/d}})$ beyond the $L^2$-phase. Here, we use the convention $\p(\infty)=\lim_{\beta\to\infty}\p(\beta)$.

\begin{corollary}\label{cor:nontriv}
In $d\geq 3$, let $\beta_{cr}^{L^{1+2/d}}\coloneqq \sup\{\beta\colon \p(\beta)>1+\frac 2d\}$.
\begin{enumerate}
 \item[(i)] If $\beta_{cr}^{L^{1+2/d}}<\infty$, then $\p(\beta)=1+\frac 2d$ for all $\beta\in[\beta_{cr}^{L^{1+2/d}},\beta_{cr}]$.
 \item[(ii)] If $\beta_{cr}^{L^{1+2/d}}<\infty$, then $\beta_{cr}^{L^{1+2/d}}>\beta_{cr}^{L^2}$.
 \item[(iii)] $\beta\mapsto\p(\beta)$ is continuous in $[\beta_{cr}^{L^2},\beta_{cr}]$.
\end{enumerate}
\end{corollary}

\begin{proof}
\textbf{Part(i)}: We first argue as in Corollary~\ref{cor:betac} to see that $\p(\beta_{cr}^{L^{1+2/d}})=1+\frac 2d$: If we assume $\p(\beta_{cr}^{L^{1+2/d}})\in(1+2/d,2]$, the right-continuity from Theorem~\ref{thm:p}(i) would imply $p(\beta)>1+2/d$ for some $\beta>\beta_{cr}^{L^{1+2/d}}$, which contradicts the definition of $\beta_{cr}^{L^{1+2/d}}$. Hence $\p(\beta_{cr}^{L^{1+2/d}})\leq 1+\frac 2d$, and from Theorem~\ref{thmx:p}(ii), Theorem~\ref{thmx:A}(i) and the monotonicity of $\p$ we obtain $\p(\beta)=1+\frac 2d$ on $[\beta_{cr}^{L^{1+2/d}},\beta_{cr}]$ .

\smallskip \textbf{Part(ii)}: By part(i) and Theorem~\ref{thmx:p}(iv), $\p(\beta_{cr}^{L^2})=2\neq 1+\frac 2d=\p(\beta_{cr}^{L^{1+2/d}})$.

\smallskip \textbf{Part (iii)}: First, be Theorem~\ref{thmx:p}(iii) and Corollary~\ref{cor:betac}, $\beta\mapsto\p(\beta)$ is left-continuous on $[0,\beta_{cr}]$. Moreover, by definition we have $\p(\beta)>1+2/d$ for $\beta<\beta_{cr}^{L^{1+2/d}}$, and together with Theorem~\ref{thm:p}(i) we find that $\beta\mapsto\p(\beta)$ is right-continuous in $[\beta_{cr}^{L^2},\beta_{cr}^{L^{1+2/d}})$. This is enough to conclude in the case $\beta_{cr}^{L^{1+2/d}}=\infty$. On the other hand, if $\beta_{cr}^{L^{1+2/d}}<\infty$, part(i) additionally shows that $\beta\mapsto\p(\beta)$ is constant, hence right-continuous, on $[\beta_{cr}^{L^{1+2/d}},\beta_{cr}]$.
\end{proof}

We refer to Figure~\ref{fig:p} for an illustration of the possible scenarios for $\p$ at $\beta_{cr}$.

\begin{corollary}\label{cor:fin}
If $d\geq 3$ and $\P$ has finite support, then $\beta\mapsto\p(\beta)$ is strictly decreasing in $[\beta_{cr}^{L^2},\beta_{cr}]$. In particular, $\p(\beta)>1+\frac 2d$ for all $\beta<\beta_{cr}$.
\end{corollary}
The assumption of finite support is of course quite restrictive and we expect that strict monotonicity holds for general environments. We note that this assumption is necessary because we rely on a quantitative hypercontractive estimate (Theorem~\ref{thmx:myhyper}) which is only available in this setup, whereas another crucial estimate (equation~\eqref{eq:compare}) is proved under the weaker assumption of bounded environment (Theorem~\ref{thmx:alea}). See also the discussion at the end of Appendix~\ref{app:hyper}.

\begin{figure}[t]
    \includegraphics[width=.66\textwidth]{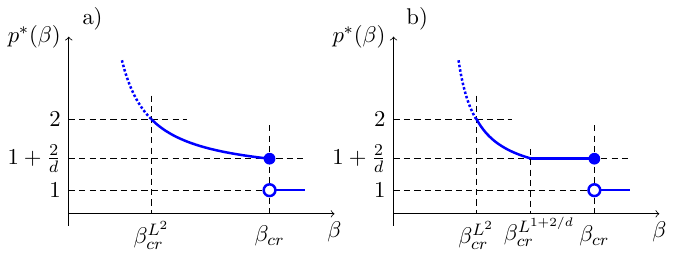}
    \caption{Classification of the behavior for $\p(\beta)$ in the case $\beta_{cr}<\infty$. We show that the only point of non-continuity beyond $\beta_{cr}^{L^2}$ is $\beta_{cr}$. Panel a) depicts the situation for finitely supported environments, which we expect to hold for general weights, and panel b) shows a situation with $\beta_{cr}^{L^{1+2/d}}<\beta_{cr}$.}\label{fig:p}
\end{figure}

\smallskip In the context of Corollary~\ref{cor:betac}, it is relevant to discuss the result from \cite{V21} about a closely related model, called the directed polymer in $\gamma$-stable environment. It is very similar to our model, except the assumption \eqref{eq:exp_mom} is dropped and we only require exponential moments up to some $\beta_0\in(0,\infty)$. It is then convenient to re-parameterize the model and consider an environment $\eta=(\eta_{t,x})_{t\in\N,x\in\Z^d}$ whose marginals are independent, centered, supported on $[-1,\infty)$ and have a Pareto-type tail,
\begin{align*}
\P(\eta_{t,x}\geq t)\sim Ct^{-\gamma}\qquad\text{ for }t\to\infty.
\end{align*}
The partition function is defined to be
\begin{align*}
\widetilde W^\beta_n\coloneqq E\Big[\prod_{t=1}^n (1+\beta \eta_{t,X_t})\Big],
\end{align*}
which is, for $\beta\in[0,1]$, a non-negative martingale. The main result in \cite{V21} is that a non-trivial phase transition occurs, i.e., $\lim_{n\to\infty}\widetilde W^\beta_n>0$ for some $\beta>0$, if and only if $\gamma>1+2/d$. Let $\widetilde p^*(\gamma,\beta)$ denote the critical exponent for $(\widetilde W^\beta_n)_{n\in\N}$, defined analogously to \eqref{eq:def_p}, which trivially satisfies $\widetilde p^*(\gamma,\beta)\leq \gamma$. The proof of \cite[Theorem~1.6]{V21} also reveals that $\widetilde p^*(\gamma,\beta)\in(1+2/d,\gamma)$ for $\gamma>1+2/d$ and $\beta>0$ small enough. In view of Corollary~\ref{cor:betac}, it is natural to conjecture that $\widetilde p^*(\gamma,\beta)\downarrow 1+\frac 2d$ as $\beta$ approaches the critical inverse temperature (for fixed $\gamma$).

\smallskip We also mention the works \cite{BL21, BL22, BCL23} concerning the existence of an \emph{intermediate weak disorder} phase for $d\geq 3$ and $\gamma\in(0,1)\cup(1,1+\frac 2d]$, which does not have a counterpart in our model since $\p(\beta)\in\{1\}\cup[1+\frac2d,\infty)$. In this regime, strong disorder holds for all $\beta>0$ and they prove that the partition function has a non-trivial limit,  called the \emph{continuum directed polymer}, if one chooses an appropriate time-dependent $\beta_n$. In contrast to the present work, the main technical difficulty arises from the small values in the environment, which are controlled by a careful martingale argument. The argument used to control the large values of the environment is somewhat similar to our techniques, compare for example \cite[Proposition~4.2]{BL22} and Lemma~\ref{lem:core}.

\subsection{Consequences of the key moment bound}\label{sec:appl}

In this section, we collect a number of results that can be proved quickly with the help of Theorem~\ref{thm:local}.

\subsubsection{Local limit theorem for the polymer measure}\label{sec:alt}

We first give a precise statement for the heuristic \eqref{eq:informal} from Section~\ref{sec:intro}.

\begin{corollary}\label{cor:informal}
	Assume $d\geq 3$ and $\p(\beta)>1+2/d$ and let $p\in(1+2/d,2\wedge \p(\beta))$. For any $\eps\in(0,\frac{3}{4})$, it holds that
	\begin{align}\label{eq:informal2}
		\lim_{r\to\infty}\sup_{n\geq r^{4(1+\eps)}}\sup_{\substack{x,y\in\Z^d\colon (0,x)\leftrightarrow(n,y),\\|x-y|\leq n^{3/4-\eps}}}\E\Big[\Big|\frac{W^{\beta,x,y}_{0,n}}{W^{\beta,x,\star}_rW^{\beta,\star,y}_{n-r,n}}-1\Big|^p\Big]=0,\\
		\lim_{r\to\infty}\sup_{n\geq r^{4(1+\eps)}}\sup_{\substack{x,y\in\Z^d\colon (0,x)\leftrightarrow(n,y),	\\|x-y|\leq n^{3/4-\eps}}}\E\Big[\Big|W^{\beta,x,y}_{0,n}-W^{\beta,x,\star}_rW^{\beta,\star,y}_{n-r,n}\Big|^p\Big]=0.\label{eq:informal3}
	\end{align}
\end{corollary}
Note that \eqref{eq:informal3} has the same form as the local limit theorem proved in \cite{V06}.

\begin{remark}
Using \eqref{eq:informal2}, we see that for $1\ll r\ll n$ and $|y|=o(n^{3/4})$,
\begin{align*}
	\frac{\mu_{\omega,n-1}^{\beta}(X_{n}=y)}{P(X_{n}=y)}&=\frac{W^{\beta,0,y}_{0,n}}{W_{n-1}^{\beta,0,\star}}=\frac{W^{\beta,0,\star}_r}{W^{\beta,0,\star}_{n-1}}\frac{W^{\beta,0,y}_{0,n}}{W_r^{\beta,0,\star}W^{\beta,\star,y}_{n-r,n}}W^{\beta,\star,y}_{n-r,n}\approx W^{\beta,\star,y}_{-\infty,n},
\end{align*}
which heuristically recovers the density that appeared in \eqref{eq:informal}.
\end{remark}
\subsubsection{Critical exponent of the pinned partition function}

Analogously to $\p$, we define
\begin{align*}
	\pp(\beta)\coloneqq \sup\left\{p\colon (W_n^{\beta, 0,0})_{n\in2\N}\text{ is $L^p$-bounded}\right\}.
\end{align*}

Theorem~\ref{thm:local}(i) immediately implies the following:
\begin{corollary}\label{cor:pp}
In any dimensions and for any $\beta\geq 0$, it holds that $\pp(\beta)\leq\p(\beta)$. If $\p(\beta)\in(1+2/d,2]$, then also $\pp(\beta)=\p(\beta)$.
\end{corollary}

\subsubsection{Stability under perturbation by small drifts}\label{sec:drift}

Next, we discuss the effect perturbing the model  by adding a small drift to the underlying random walk. For $\llambda\in\R^d$, we let $P^\llambda$ be the random walk with increment distribution
\begin{align*}
	P^{\llambda}(X_{k+1}=x+y|X_k=x)=\1_{|y|_1=1}\frac{1}{2d}e^{\llambda \cdot y-\varphi(\llambda)},
\end{align*}
where $\varphi(\llambda)=\log(\frac{1}{2d}\sum_{z\colon |z|_1=1}e^{\llambda\cdot z})$. The drift is denoted by
\begin{align}\label{eq:def_m}
	\m(\llambda)=E^{\llambda}[X_1]=\frac{1}{2d}\sum_{y\colon |y|_1=1}ye^{\llambda \cdot y-\varphi(\llambda)}.
\end{align}

Note that $\m(\llambda)\sim \frac{\llambda}d$ as $\llambda\to 0$. We write $P^{\llambda,x,\star}_{m,n}$ and $P^{\llambda,\star,y}_{m,n}$ for the laws of $(X_k-X_m+x)_{m\leq k\leq n}$ and $(X_k-X_n+y)_{m\leq k\leq n}$ under $P^\llambda$, i.e., the random walks going forward and backward from a given space-time point. Note that $E^{\llambda,x,\star}_{m,n}[X_n]=x+(n-m)\m(\llambda)$ and $E^{\llambda,\star,y}_{m,n}[X_m]=y-(n-m)\m(\llambda)$. We furthermore write $W^{\beta,\llambda,x,\star}_{m,n}$ and $W^{\beta,\llambda,\star,y}_{m,n}$ for the corresponding partition function and reverse partition functions, with $W_n^{\beta,\llambda}\coloneqq W^{\beta,\llambda,0,\star}_{0,n}$. The critical exponent is generalized as follows:
\begin{align*}
	\p(\beta,\llambda)\coloneqq \sup\big\{p\colon (W_n^{\beta,\llambda})_{n\in\N}\text{ is $L^p$-bounded}\big\}.
\end{align*}

\begin{corollary}\label{cor:drift}
Assume $d\geq 3$ and $\p(\beta)\in(1+2/d,2]$. Then $\llambda\mapsto \p(\beta,\llambda)$ is lower-semicontinuous at $\0$, that is,
\begin{align*}
	\lim_{\eps\downarrow 0}\inf_{\llambda\in[-\eps,\eps]^d}\p(\beta,\llambda)\geq \p(\beta,\0).
\end{align*}
In particular, there exists $\lambda_0$ such that weak disorder holds in $[-\lambda_0,\lambda_0]^d$, i.e.,
\begin{align}\label{eq:weak_stable}
	\lim_{n\to\infty}W^{\beta,\llambda}_n>0\qquad\text{ for all }\llambda\in[-\lambda_0,\lambda_0]^d.
\end{align}
\end{corollary}

Corollary~\ref{cor:drift} yields the following consequence for the large deviation rate function. For background on the theory of large deviations we refer to, for example,  \cite{DZ}.

\begin{corollary}\label{cor:LDP}
In any dimension and for any $\beta$, the polymer measure $(\mu_{\omega,n}^\beta(X_n/n\in\cdot))_{n\in\N}$ satisfies a quenched large deviation principle with deterministic, convex, good, rate function $I^\beta$. Moreover, in $d \geq 3$ and in weak disorder it holds that $I^\beta\geq I^{0}$, where $I^{0}$ is the rate function of the simple random walk. If additionally $\p(\beta)>1+2/d$, then there exists $\eps>0$ such that $I^\beta|_{[-\eps,\eps]^d}\equiv I^{0}|_{[-\eps,\eps]^d}$.
\end{corollary}

The equality of $I^\beta$ and $I^0$ in a neighborhood of the origin was first noted in \cite[Exercise~9.1]{C17} in the $L^2$-phase. It was extended to the interior of the whole weak disorder phase in two closely related models in \cite{FJ20, J21_2}.

\subsubsection{Replica overlap and delocalization}

We close this section by discussing some consequences for the replica overlap $I_n^{\beta,2}$ introduced before Theorem~\ref{thmx:B}. This quantity appears naturally in the Doob decomposition of $(\log W_n)_{n\in\N}$, see \cite[Section 2]{CSY03}, and this connection is the basis for the proof of Theorem~\ref{thmx:B}(ii). Note that the family $(\mu_{\omega,k}^\beta)_{k=1,\dots,n}$ is not consistent in the sense of Kolmogorov’s extension theorem, see for example \cite[A.3.1]{durrett}, so it is not possible to interpret the expression $\sum_{k=1}^nI_k^{\beta,2}$ as the expected overlap between two polymers up to time $n$. For that reason, it is more natural to consider a fixed time horizon, i.e., $\sum_{k=1}^nI_{k,n}^{\beta,2}$ with

\begin{align*}
	I_{k,n}^{\beta,2}\coloneqq \sum_{x\in\Z^d}\mu_{\omega,n}^{\beta}(X_k=x)^2.
\end{align*}
This object cannot be analyzed with the approach based on the Doob decomposition, although we note that some estimates have been obtained in \cite{CC13} in a related model with Gaussian disorder with the help of Mallivian calculus.

\smallskip In another direction, it is natural to wonder about the summability of $I_n^{\beta,p}$, defined by
\begin{align*}
	I_n^{\beta,p}\coloneqq \sum_{x\in\Z^d}\mu_{\omega,n}^{\beta}(X_{n+1}=x)^p.
\end{align*}
There is no interpretation of $I_n^{\beta,p}$ in terms of replicas as in the case $p=2$, but the rate of decay of $I_n^{\beta,p}$ is closely related to the localization or delocalization of $\mu_{\omega,n}^\beta$. With the help of Theorem~\ref{thm:local}, we obtain the following:

\begin{corollary}\label{cor:I}
Assume $d\geq 3$ and $\p(\beta)>1+2/d$.
\begin{enumerate}
	\item [(i)] For any $p>1+2/d$, it holds that $\sum_nI_n^{\beta,p}<\infty$ almost surely.
	\item[(ii)] It holds that $\sup_{n\in\N}\sum_{k=1}^n I_{k,n}^{\beta,2}<\infty$ almost surely.
\end{enumerate}
\end{corollary}

Corollary~\ref{cor:I}(i) places some restrictions on the rate of decay of $\max_{x\in\Z^d}\mu_{\omega,n}^{\beta}(X_{n+1}=x)$, simultaneously for all $n$ large enough.

\smallskip As a final consequence of Theorem~\ref{thm:local}, we prove an upper bound on $\max_{x\in\Z^d}\mu_{\omega,n}^\beta(X_{n+1}=x)$ for typical $n$.

\begin{corollary}\label{cor:typical}
	Assume $d\geq 3$ and $\p(\beta)>1+2/d$. For any $\eps>0$, it holds that
\begin{equation}\label{eq:aasa}
\begin{split}
	&\lim_{n\to\infty}\P\Big(\max_{x\in\Z^d}\mu_{\omega,n}^{\beta}(X_{n+1}=x)\geq  n^{-\frac{d}{2}(1-\frac{1}{\p\wedge 2})+\eps}\Big)\\&=\lim_{n\to\infty}\P\Big((I_n^{\beta,2})^{1/2}\geq n^{-\frac{d}{2}(1-\frac{1}{\p\wedge 2})+\eps}\Big)=0.
\end{split}
\end{equation}
\end{corollary}
We believe that it possible to prove the following lower bound:
\begin{align*}
	\lim_{n\to\infty}\P\Big(\max_{x\in\Z^d}\mu_{\omega,n}^\beta(X_{n+1}=x)\leq  n^{-\frac{d}{2}(1-\frac 1\p)-\eps}\Big)=\lim_{n\to\infty}\P\Big((I_n^{\beta,2})^{1/2}\leq n^{-\frac{d}{2}(1-\frac 1\p)-\eps}\Big)=0.
\end{align*}
In particular, we think that \eqref{eq:aasa} is sharp outside of the $L^2$-phase.
On the other hand, we expect that the situation is different in $L^2$-weak disorder, where it should be the case that, in probability,
\begin{align*}
	\lim_{n\to\infty}-\frac{1}{\log n}\max_{x\in\Z^d}\log \mu_{\omega,n}^\beta(X_n=x)=\frac{d}{2}\Big(1-\frac{1}{\p}\Big)>\frac{d}{4}=\lim_{n\to\infty}-\frac 1{\log n}\log (I_n^{\beta,2})^{1/2}.
\end{align*}

Note that the conjectured rate of decay $\max_x\mu_{\omega,n}^\beta(X_n=x)\approx n^{-\frac d2(1-\frac 1{\p})+o(1)}$ is slower than $n^{-d/2}$, which is the corresponding rate for simple random walk. This is a reflection of the fact that while the polymer measure behaves like the simple random walk macroscopically due to Theorem~\ref{thmx:A}(i), the microscopic structure is quite different and governed by local fluctuations in the environment. Since the techniques for lower bounds are quite different, we leave this direction of research for future work.

\section{Auxiliary results}\label{sec:aux}

Let $P^{\otimes 2,\llambda}$ denote the law of two independent random walks $(X_n)_{n\in\N}$ and $(X_n')_{n\in\N}$ with marginals $P^\llambda$. We write $P^{\llambda,x}$ and $P^{\otimes 2,\llambda,(x,x')}$ to indicate the starting points and if $\mu$ is a probability measure on $\Z^d$, we write
\begin{align*}
P^{\llambda,\mu}=\sum_x\mu(x)P^{\llambda,x}\qquad\text{ and }\qquad P^{\otimes 2,\llambda,\mu}=\sum_{x,x'}\mu(x)\mu(x')P^{\otimes 2,\llambda,(x,x')}.
\end{align*}

\begin{lemma}\label{lem:core}
	Let $p>1+\frac 2d$. There exists $C>0$ such that, for all $\mu\in\M_1(\Z^d)$ and $\llambda\in[-1,1]^d$,
	\begin{align}
		&\sum_{t\in\N,z\in\Z^d}P^{\otimes 2,\llambda,\mu}\big(X_s\neq X_s'\text{ for }s=0,\dots,t-1,X_t=X_t'=z\big)^{p/2} \leq C\max_x \mu(x)^{p-(1+2/d)}.\label{eq:core}
	\end{align}
	Moreover, for all $T\in\N$ it holds that
	\begin{align}
	&\sum_{t\geq T,z\in\Z^d}P^{\otimes 2,\llambda,\mu}\big(X_s\neq X_s'\text{ for }s=0,\dots,t-1,X_t=X_t'=z\big)^{p/2}\leq CT^{-\frac d2(p-1)+1}.\label{eq:coreT}
	\end{align}
\end{lemma}
\begin{proof}[Proof of Lemma~\ref{lem:core}]
We bound
\begin{align*}
&	P^{\otimes 2,\llambda,\mu}(X_s\neq X_s'\text{ for all }s=0,\dots,s-1,X_t=X_t'=z)^{p/2}\\
&\quad \leq P^{\otimes2,\llambda,\mu}(X_t=X_t'=z)^{p/2}\\
&\quad =((\mu * P_t^{\llambda})(z))^p,
\end{align*}
where $P_t^{\llambda}(x)=P^{\llambda}(X_t=x)$ and ``$*$'' denotes the convolution of two functions $\Z^d\to\R$. Writing $\|f\|_r=(\sum_{x\in\Z^d} |f(x)|^r)^{1/r}$, Young's convolution inequality now gives
\begin{align}\label{eq:cases}
	\|\mu *P_t^{\llambda}\|_p^p \leq\min\big\{\|\mu\|_1^p\|P_t^{\llambda}\|_p^p,\|\mu\|_p^p\|P_t^{\llambda}\|_1^p\big\} =\min\big\{\|P_t^{\llambda}\|_p^p,\|\mu\|_p^p\big\},
\end{align}
where we used that both $\mu$ and $P_t^\llambda$ are probability measures\footnote{As pointed out by an anonymous referee, \eqref{eq:cases} also follows from Jensen's inequality.}. By the local limit theorem for $P^\llambda$, Theorem~\ref{thmx:local} in the appendix, there exists $C>0$ such that, for all $\llambda\in[-1,1]^d$,
\begin{align}\label{eq:summable}
	\|P_t^{\llambda}\|_p^p=\sum_x P^{\llambda}_t(x)^p\leq \max_x P_t^{\llambda}(x)^{p-1}\sum_x P_t^{\llambda}(x)=\max_x P_t^{\llambda}(x)^{p-1}\leq Ct^{-\frac d2(p-1)}.
\end{align}
By assumption, $\frac d2(p-1)>1$, so this bound is summable and \eqref{eq:coreT} follows. For \eqref{eq:core}, we fix $T\in\N$ and use the bound $\|\mu\|_p^p$ from \eqref{eq:cases} for $t\leq T$ and the bound $\|P_t^\llambda\|_p^p$ for $t>T$, which gives
\begin{align*}
	&\sum_{t\in\N,z\in\Z^d} P^{\otimes 2,\llambda,\mu}\left(X_s\neq X_s'\text{ for all }s=0,\dots,s-1,X_t=X_t'=z\right)^{p/2}\\
	&\quad\leq T\max_x\mu(x)^{p-1}+CT^{-\frac d2(p-1)+1}.
\end{align*}
We have used that $\|\mu\|_p^p\leq \max_x \mu(x)^{p-1}$. Inserting $T=(\max_x \mu(x))^{-2/d}$ now gives \eqref{eq:core}.
\end{proof}

Next, we  obtain a bound for a certain joint partition function. Similar to \eqref{eq:def_joint_restricted}, let
\begin{align*}
	W^{\otimes 2,\beta,\llambda,(x,x')}_{n}[\1_B]=E^{\otimes 2,\llambda,(x,x')}\big[e^{\beta H_{(0,n]}(\omega,X)+\beta H_{(0,n]}(\omega,X')-2n\lambda(\beta)}\1_B\big].
\end{align*}

\begin{lemma}\label{lem:core2}Assume $\p(\beta)>1+2/d$ and let $p\in(1+2/d,\p(\beta)\wedge 2)$. For any $\eps>0$ there exist $\beta_0>\beta$, $\lambda_0=\lambda_0(\beta)>0$ and $T\in\N$ such that, for all  $\beta'\in[\beta,\beta_0]$ and $\llambda\in[-\lambda_0,\lambda_0]^d$,
	\begin{align}
		\sum_{t\geq T,z\in\Z^d}\E\Big[W_t^{\otimes 2,\beta',\llambda}\big[\1_{X_s\neq X_s'\text{ for all }s=T,\dots,t-1,X_t=X_t'=z}\big]^{p/2}\Big]\leq \eps\label{eq:first}.
	\end{align}
	Furthermore, there exists $C>0$ such that, for all $n\geq T$, $\beta'\in[\beta,\beta_0]$ and $\llambda\in[-\lambda_0,\lambda_0]^d$,
	\begin{align}
		\sum_{t\geq n,z\in\Z^d}\E\Big[W_t^{\otimes 2,\beta',\llambda}\big[\1_{X_s\neq X_s'\text{ for all }s=T,\dots,t-1,X_t=X_t'=z}\big]^{p/2}\Big]\leq Cn^{-(p-(1+2/d))}\label{eq:second}.
	\end{align}
\end{lemma}
\begin{proof}
Using the conditional Jensen inequality, we get for any $\beta'\in\R_+$ and $\llambda\in\R^d$,
\begin{align*}
&\E\left[W_t^{\otimes 2,\beta',\llambda}\big[\1_{X_s\neq X_s'\text{ for all }s=T,\dots,t-1,X_t=X_t'=z}\big]^{p/2}\right]\\
&\leq \E\left[\E\left[W_t^{\otimes 2,\beta',\llambda}\big[\1_{X_s\neq X_s'\text{ for all }s=T,\dots,t-1,X_t=X_t'=z}\big]\Big|\F_{T-1}\right]^{p/2}\right]\\
&=c\E\left[W_{T-1}^{\otimes 2,\beta',\llambda}\left[\1_{X_s\neq X_s'\text{ for all }s=T,\dots,t-1,X_t=X_t'=z}\right]^{p/2}\right]\\
&=c\E\Big[\big(W_{T-1}^{\beta',\llambda}\big)^{p}P^{\otimes 2,\llambda,\mu_{\omega,T-1}^{\beta',\llambda}}\big(X_s\neq X_s'\text{ for all }s=0,\dots,t-T-1,X_{t-T}=X_{t-T}'=z\big)^{p/2}\Big],
\end{align*}
where $c=e^{\lambda(p\beta')-p\lambda(\beta')}$ and $\mu_{\omega,T-1}^{\beta',\llambda}(\cdot)=\mu_{\omega,T-1}^{\beta',\llambda}(X_T\in\cdot)$ is the biased polymer measure, defined similarly to \eqref{eq:mu} by
\begin{align*}
	\mu_{\omega,T-1}^{\beta',\llambda}(X_{T}\in A)=\frac{E^{\llambda}[e^{\beta' H_{T-1}(\omega,X)-(T-1)\lambda(\beta')}\1_{X_{T}\in A}]}{W^{\beta',\llambda}_{T-1}}.
\end{align*}Summing over $t$ and $z$ and applying Lemma~\ref{lem:core}, we get
\begin{align}
	&\sum_{t\geq T,z\in\Z^d}\E\left[W^{\otimes2,\beta',\llambda}_{t}\left[\1_{X_s\neq X_s'\text{ for all }s=T,\dots,t-1,X_t=X_t'=z}\right]^{p/2}\right]\label{eq:herehere}\\
	&\quad\leq Ce^{\lambda(p\beta')-p\lambda(\beta')}\E\Big[\big(W^{\beta',\llambda}_{T-1}\big)^{p} \max_x \mu^{\beta',\llambda}_{\omega,T-1}(x)^{p-(1-2/d)}\Big].\label{eq:here}
\end{align}
We first consider this equation for $\beta'=\beta$ and $\llambda=\0$. Let $\delta>0$ be small enough that $p(1+\delta)<\p(\beta)$ and let $q$ be the H\"older dual of $1+\delta$, so that
\begin{align*}
&	\E\Big[\big(W^{\beta,\0}_{T-1}\big)^p \sup_x \mu^{\beta,\0}_{\omega,T-1}(X_T=x)^{p-(1-2/d)}\Big]\\
&\leq \E\Big[\big(W^{\beta,\0}_{T-1}\big)^{p(1+\delta)}\Big]^{1/(1+\delta)}\E\Big[ \sup_x \mu^{\beta,\0}_{\omega,T-1}(X_T=x)^{q(p-(1-2/d))}\Big]^{1/q}.
\end{align*}
By assumption, the first factor is bounded in $T$. Moreover, $\sup_x\mu_{\omega,T-1}^{\beta,\0}(X_T=x)\leq (I_{T-1}^{\beta,2})^{1/2}$ (recall \eqref{eq:def_I}) and thus the supremum in the second expectation converges to zero almost surely by Theorem~\ref{thmx:B}(i). Since $I_{T-1}^{\beta,2}\leq 1$, the convergence to zero also holds in $L^q$, for any $q\in[1,\infty)$. Thus we find $T\in\N$ such that \eqref{eq:here} is bounded by $\frac{\eps}{2}$ for $\beta'=\beta$ and $\llambda=\0$. Note that the quantity inside the expectation in \eqref{eq:here} is a continuous function of $\beta'$ and $\llambda$  (since the time-horizon $T-1$ is fixed and the random walk has finite range). Thus, by the dominated convergence theorem, we can choose $\beta_0$ and $\lambda_0$ small enough that the right-hand side of \eqref{eq:here} is bounded by $\eps$ for all $\beta'\in[\beta,\beta_0]$ and $\llambda\in[-\lambda_0,\lambda_0]^d$, as desired.

\smallskip Finally, with the same values of $T$, $\beta_0$ and $\lambda_0$ but  applying \eqref{eq:coreT} instead of \eqref{eq:core} in \eqref{eq:herehere}, we get
\begin{align*}
	&\sum_{t\geq n,z\in\Z^d}\E\left[W^{\otimes2,\beta',\llambda}_{t}\big[\1_{X_s\neq X_s'\text{ for all }s=T,\dots,t-1,X_t=X_t'=z}\big]^{p/2}\right]\\
	&\quad\leq Ce^{\lambda(p\beta')-p\lambda(\beta')}n^{-(p-(1+2/d))}\E\Big[\big(W^{\beta',\llambda}_{T-1}\big)^{p} \Big].
\end{align*}
The claim follows since the quantities in the final line are bounded in $\beta'\in[\beta,\beta_0]$ and $\llambda\in[-\lambda_0,\lambda_0]^d$, for fixed $T$. \end{proof}

In the next lemma, we consider an expectation with respect to the random walk bridge between times $0$ and $n$ where the integrand does not depend on what happens in the ``middle''. We show that, in this case, the expectation can be factorized into two expectations with respect to the simple random walk. To make this precise, we write $\G_I\coloneqq \sigma(X_t\colon i\in I\cap\N)$ and $\G_I^{\otimes 2}\coloneqq \sigma(X_t,X_t'\colon t\in I\cap \N)$ for the filtration of the walks. We further write $f\in\G_I$ if $f$ is a $\G_I$-measurable real-valued function, and similarly for $f\in\G^{\otimes 2}_I$.

\begin{lemma}\label{lem:lclt}
	\begin{itemize}
		\item[(i)]For all $\eps\in(0,\frac{1}{2})$ and $M>1$, there exists $C=C_{\eps,M}>0$ such that, for all $n\in\N$ and $\llambda\in[-1,1]^d$,
\begin{align*}
	\sup_{\substack{x,x',y,y'\in\Z^d\\(0,x)\leftrightarrow(n,y),(0,x')\leftrightarrow(n,y')\\|x-y-n\m(\llambda)|\leq Mn^{1/2}\\|x'-y'-n\m(\llambda)|\leq Mn^{1/2}}}\sup_{\substack{0\leq s\leq t\leq n\\t-s\geq \eps n}}\sup_{\substack{f\in\G^{\otimes 2}_{[1,s]}\\g\in\G^{\otimes 2}_{[t,n]}}}\frac{E^{\otimes 2,\llambda,(x,x'),(y,y')}_{0,n}[fg]}{E^{\otimes 2,\llambda,(x,x'),\star}_{0,n}[f]E^{\otimes 2,\llambda,\star,(y,y')}_{0,n}[g]}\leq C,
\end{align*}
\item[(ii)] For every $\eps\in(0,\frac{1}{4})$, there exists a non-negative sequence $(a_n)_{n\in\N}$ with $\lim_{n\to\infty}a_n=0$ such that, for all $n\in\N$,
	\begin{align}\label{eq:display}
\sup_{\substack{x,y\in\Z^d\\(0,x)\leftrightarrow(n,y)\\|x-y|\leq n^{3/4-\eps}}}\sup_{\substack{s\leq n^{{1}/{4}-\eps}\\t\geq n-n^{{1}/{4}-\eps}}}\sup_{f\in\G_{[1,s]},g\in\G_{[t,n]}}\Big|\frac{E^{x,y}_{0,n}[fg]}{E^{x,\star}_{0,n}[f]E^{\star,y}_{0,n}[g]}-1\Big|\leq a_n.
\end{align}\end{itemize}
\end{lemma}
\begin{proof}
We start with \textbf{part (i)}: for any $z_1,z_2,z_1',z_2'\in\Z^d$, we have
\begin{align*}
	&E^{\otimes 2,(x,x'),(y,y')}_{0,n}[fg\1_{(X_s,X_s')=(z_1,z_1')}\1_{(X_t,X_t')=(z_2,z_2')}]\notag\\
	&=\frac{E^{\otimes 2,(x,x'),\star}_{0,n}[fg\1_{(X_s,X_s')=(z_1,z_1')}\1_{(X_t,X_t')=(z_2,z_2')}\1_{(X_n,X_n')=(y,y')}]}{P^{\otimes 2}((X_n,X_n')=(y-x,y'-x'))}\notag\\
	&=\frac{E^{\otimes 2,(x,x'),\star}_{0,n}[fg\1_{(X_s,X_s')=(z_1,z_1')}\1_{(X_t,X_t')=(z_2,z_2')}\1_{(X_n,X_n')=(y,y')}e^{\llambda(X_n-x)+\llambda(X_n'-x')-2n\varphi(\llambda)}]}{E^{\otimes 2}[\1_{(X_n,X_n')=(y-x,y'-x')}e^{\llambda X_n+\llambda X_n'-2n\varphi(\llambda)}]}\notag\\
	&=\frac{E^{\otimes 2,\llambda,(x,x'),\star}_{0,n}[fg\1_{(X_s,X_s')=(z_1,z_1')}\1_{(X_t,X_t')=(z_2,z_2')}\1_{(X_n,X_n')=(y,y')}]}{P^{\otimes 2,\llambda}((X_n,X_n')=(y-x,y'-x'))}\notag\\
	&=E^{\otimes 2,\llambda,(x,x'),\star}_{0,n}[f\1_{(X_s,X_s')=(z_1,z_1')}]E^{\otimes 2,\llambda,\star,(y,y')}_{0,n}[g\1_{(X_t,X_t')=(z_2,z_2')}]\notag\\
	&\qquad\times \frac{P^{\otimes 2,\llambda}((X_{t-s},X_{t-s}')=(z_2-z_1,z_2'-z_1'))}{P^{\otimes 2,\llambda}((X_n,X_n')=(y-x,y'-x'))}\\
	&\leq C E^{\otimes 2,\llambda,(x,x'),\star}_{0,n}[f\1_{(X_s,X_s')=(z_1,z_1')}]E^{\otimes 2,\llambda,\star,(y,y')}_{0,n}[g\1_{(X_t,X_t')=(z_2,z_2')}],\notag
\end{align*}
where the inequality is due to Theorem~\ref{thmx:local}. The claim follows by summing over $z_1,z_1',z_2,z_2'$.

\smallskip For \textbf{part (ii)}, by repeating the above calculation with a single random walk and $\llambda=\0$, we see that it is enough to show that, uniformly over all $x,y,s,t$ as in \eqref{eq:display} and all $z_1\in x+[-s^{\frac{1}{4}-\eps},s^{\frac{1}{4}-\eps}]^d,z_2\in y+[-(n-t)^{\frac{1}{4}-\eps},(n-t)^{\frac{1}{4}-\eps}]^d$,
\begin{align}\label{eq:cons}
	\frac{P(X_{t-s}=z_2-z_1)}{P(X_n=y-x)}=1+o(1)\qquad\text{ as }n\to\infty.
\end{align}
We will prove this with the help of a stronger version of the local limit theorem for the simple random walk, namely \cite[Theorem~2.3.11]{LL10}. Note that the result is not directly applicable to $P$ since the assumption ``$p\in\mathcal P_d$'' requires the random walk to be aperiodic. This can be rectified as follows: for even $n$, we can combine two steps into one and consider $Y_n\coloneqq X_{2n}/2$, which defines an irreducible and aperiodic random walk on $\Z^d$. Furthermore, by changing $s$ to $s+1$ and $t$ to $t-1$ if necessary, we can also assume $s$ and $t$ to be even. For odd $n$, we can decompose
\begin{align*}
	E^{x,y}_{0,n}[fg]&=\sum_{|y-y'|_1=1}P^{x,y}_{0,n}(X_{n-1}=y') E^{x,y'}_{0,n-1}[fg^{y'}],\\
	E^{\star,y}_{0,n}[f]&=\sum_{|y-y'|_1=1}P^{\star,y}_{0,n}(X_{n-1}=y')E^{\star,y'}_{0,n-1}[g^{y'}].
\end{align*}
with appropriately defined $g^{y'}\in\G_{[t,n-1]}$, $|y'-y|_1=1$.

Let now $p_t(x)\coloneqq \frac{1}{(2\pi n)^{d/2}}e^{-d|x|^2/2n}$ and note that, by \cite[Theorem~2.3.11]{LL10},
\begin{align*}
	P(X_{t-s}=z_2-z_1)&=p_{t-s}(z_2-z_1)e^{O(\frac{1}{t-s}+\frac{|z_2-z_1|^4}{(t-s)^3})},\\
	P(X_n=y-x)&=p_n(y-x)e^{O(\frac{1}{n}+\frac{|y-x|^4}{n^3})}
\end{align*}
Due to our assumptions on $x,y,z_1$ and $z_2$, the exponential terms  converge to one uniformly. Moreover, since
\begin{align*}
	|y-x|^2\Big|\frac{1}{n}-\frac{1}{t-s}\Big|&\leq
	n^{-\frac 14-3\eps},\\
	\frac{1}{t-s}\big||y-x|^2-|z_2-z_1|^2\big|
				     &\leq 4n^{-2\eps},
\end{align*}
we see that
\begin{align*}
	\frac{p_{t-s}(z_2-z_1)}{p_n(y-x)}=\Big(\frac{n}{t-s}\Big)^{d/2}e^{d|y-x|^2(\frac{1}{2n}-\frac{1}{2(t-s)})}e^{\frac{d}{2(t-s)}(|y-x|^2-|z_2-z_1|^2)}
\end{align*}
also converges to one uniformly.
\end{proof}

\section{The decomposition of joint partition functions}\label{sec:chaos}

To implement the idea outlined in Section~\ref{sec:strat}, we introduce a sequence of stopping times for two paths $(X_n,X_n')_{n\in\N}$. See also Figure~\ref{fig:constr} for an illustration. For fixed $T\in\N$, let
	\begin{equation}\label{eq:decomposition}\begin{split}
		\tau_0&\coloneqq \inf\{k\geq 0\colon X_k=X_k'\},\\
		\tau_{k+1}&\coloneqq \inf\{k\geq  \tau_k+T\colon X_k=X_k'\}\qquad\text{ for } k\geq 0.
	\end{split}\end{equation}
Note that $\tau_0$ is not required to be larger than $T$, so for example $\tau_0=0$ if $X_0=X_0'$. Next, let
\begin{align}
	K_n&\coloneqq \inf\{k\colon \tau_k> n\},\label{eq:def_K}\\
	L_n&\coloneqq \#\{0< i< K_n\colon \tau_i-\tau_{i-1}\geq n^{1/2}\}.
\end{align}
That is, $K_n$ is the number of collisions (separated in time by at least $T$) and $L_n$ is the number of ``large'' gaps between collisions. The value $n^{1/2}$ in the definition of $L_n$ is arbitrary, we could have used any exponent in $(0,1)$.  Note that in the definition of $L_n$, the first interval $[0,\tau_0]$ and last interval $[\tau_{K_n-1},n]$ are not counted as ``large'' even if they are longer than $n^{1/2}$.

\begin{figure}[h]
	\includegraphics[width=.8\textwidth]{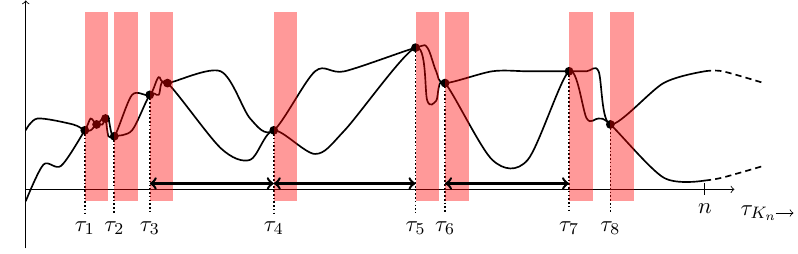}
\caption{Illustration of the collision times between two paths. The shaded area indicates the time intervals of the form $[\tau_k,\tau_k+T)$, during which any new collision will be ignored. In this case, we have $K_n=9$ and there are $L_n=3$ large intervals, which are marked by the bold arrows at the bottom.  \label{fig:constr}}
\end{figure}

\smallskip Note further that $\tau_0,\tau_1,\dots$ depend on $T$, which is chosen as follows:
\begin{lemma}\label{lem:chaos}
	Assume $\p(\beta)>1+2/d$ and let $p\in(1+2/d,\p(\beta)\wedge 2)$. Let $\lambda_0$, $\beta_0$ and $T$ be as in Lemma~\ref{lem:core2} with $\eps=\frac{1}{4}$. There exist $C,C'>0$ such that, for all $n\in\N$, $k,l\geq 1$, $\beta'\in[\beta,\beta_0]$, $\llambda\in[-\lambda_0,\lambda_0]^d$ and $\mu\in\M_1(\Z^d)$,
\begin{align}
\E\big[W^{\otimes 2,\beta',\llambda,\mu,\star}_{0,n}[\1_{K_n=k}]^{p/2}\big]&\leq C4^{-k}\max_x\mu(x)^{p-(1+2/d)},\label{eq:K}\\
\E\big[W^{\otimes 2,\beta',\llambda,\mu,\star}_{0,n}[\1_{L_n\geq l}]^{p/2}\big]&\leq C \left(C'n^{-(p-(1+2/d))/2}\right)^l.\label{eq:L}
\end{align}
In particular, it holds that
\begin{align}\label{eq:KK}
\sup_{n\in\N}\sup_{\llambda\in[-\lambda_0,\lambda_0]^d}\sup_{\beta'\in[\beta,\beta_0]}\E\big[(W^{\beta',\llambda}_n)^p\big]<\infty.
\end{align}
\end{lemma}

\begin{proof}[Proof of Lemma~\ref{lem:chaos}]
We decompose
	\begin{align}
		&\E\left[W^{\otimes 2,\beta',\llambda,\mu,\star}_{0,n}[\1_{K_n=k}]^{p/2}\right]\notag\\
		&=\E\Big[\Big(\sum_{\substack{t_0,\dots,t_{k-1}\\z_0,\dots,z_{k-1}}} W^{\otimes 2,\beta',\llambda,\mu,\star}_{0,n}\Big[\1_{K_n=k,\tau_i=t_i, X_{\tau_i}=z_i\text{ for all }i=0,\dots,k-1}\Big]\Big)^{p/2}\Big]\notag\\
		&\leq	\sum_{\substack{t_0,\dots,t_{k-1}\\z_0,\dots,z_{k-1}}}\E\Big[\Big( W^{\otimes 2,\beta',\llambda,\mu,\star}_{0,n}\Big[\1_{K_n=k,\tau_i=t_i, X_{\tau_i}=z_i\text{ for all }i=0,\dots,k-1}\Big]\Big)^{p/2}\Big],\label{eq:insert}
	\end{align}
	where the summation is over $z_0,\dots,z_{k-1}\in\Z^d$ and $0\leq t_0\leq t_1\leq\dots t_{k-1}\leq n$ with $t_i\geq t_{i-1}+T$, $i=1,\dots,k-1$. We have used \eqref{eq:sa} in the inequality. Next, we observe
\begin{equation}\label{eq:observe}
\begin{split}
	&W^{\otimes 2,\beta',\llambda,\mu,\star}_{0,n}\big[\1_{K_n=k,\tau_i=t_i, X_{\tau_i}=X_{\tau_i}'=z_i\text{ for all }i=0,\dots,k-1}\big]\\
	&=W^{\otimes 2,\beta',\llambda,\mu,\star}_{0,t_0}\big[\1_{\tau_0=t_0,X_{\tau_0}=z_0}\big]W^{\otimes 2,\beta',\llambda,(z_{k-1},z_{k-1}),\star}_{t_{k-1},n}\big[\1_{X_s\neq X_s'\text{ for all }s=t_{k-1}+T,\dots,n}\big]\\
	&\quad\times\prod_{i=1}^{k-1}W^{\otimes 2,\beta',\llambda,(z_{i-1},z_{i-1}),\star}_{t_{i-1},t_i}\big[\1_{X_s\neq X_s'\text{ for all }s=t_{i-1}+T,\dots,t_i-1,X_{t_i}=X_{t_i}'=z_i}\big].
\end{split}
\end{equation}
We note that since the quantities on the right-hand side depend on distinct subsets of the environment, they are independent. By Jensen's inequality, the first term can be estimated as
\begin{align*}
	\E\left[W^{\otimes 2,\beta',\llambda,\mu,\star}_{0,t_0}\big[\1_{\tau_0=t_0,X_{\tau_0}=z_0}\big]^{p/2}\right]&\leq \E\left[W^{\otimes 2,\beta',\llambda,\mu,\star}_{0,t_0}\big[\1_{\tau_0=t_0,X_{\tau_0}=z_0}\big]\right]^{p/2}\\
											 &=e^{\lambda(p\beta')-p\lambda(\beta')}P^{\otimes 2,\llambda,\mu,\star}(\tau_0=t_0,X_{\tau_0}=z_0)^{p/2},
\end{align*}
where we have used that, by definition, $X_s\neq X_s$ for $s=0,\dots,\tau_0-1$. For the last term we similarly get
\begin{align*}
	&\E\Big[W^{\otimes 2,\beta',\llambda,(z_{k-1},z_{k-1}),\star}_{t_{k-1},n}\big[\1_{X_s\neq X_s'\text{ for all }s=t_{k-1}+T,\dots,n}\big]^{p/2}\Big]\\
	&\leq\E\left[\E\Big[W^{\otimes 2,\beta',\llambda,(z_{k-1},z_{k-1}),\star}_{t_{k-1},n}\big[\1_{X_s\neq X_s'\text{ for all }s=t_{k-1}+T,\dots,n}\big]\Big|\F_{t_{k-1}+T-1}\Big]^{p/2}\right]\\
	&\leq\E\left[W^{\otimes 2,\beta',\llambda,(z_{k-1},z_{k-1}),\star}_{t_{k-1},t_{k-1}+T-1}\big[\1_{X_s\neq X_s'\text{ for all }s=t_{k-1}+T,\dots,n}\big]^{p/2}\right]\\
	&\leq\E\Big[\big(W^{\beta',\llambda}_{T-1}\big)^{p}\Big].
\end{align*}
By dropping the restriction ``$t_{k-1}\leq n$'' in \eqref{eq:insert} and using the independence noted after \eqref{eq:observe}, we obtain
\begin{align}&	\E\left[W^{\otimes 2,\beta',\llambda,\mu,\star}_n[\1_{K_n=k}]^p\right]	 \leq e^{\lambda(p\beta')-p\lambda(\beta')}\E\Big[\big(W_{T-1}^{\beta',\llambda}\big)^{p}\Big]\label{eq:middle1}
\\
&\times \Big(\sum_{t_0\in\N,z_0\in\Z^d} P^{\otimes 2,\llambda,\mu,\star}\big(X_s\neq X_s'\text{ for all }s=0,\dots,t_0-1,X_{t_0}=X_{t_0}'=z_0\big)^{p/2}\Big)\label{eq:middle2}
\\
											&\times \Big(\sum_{t\geq T,z\in\Z^d} \E\left[W^{\otimes 2,\beta',\llambda,\star}_t\big[\1_{X_s\neq X_s'\text{ for all }s=T,\dots,t-1,X_t=X_t'=z}\big]^{p/2}\right]\Big)^{k-1}.\label{eq:middle3}
\end{align}
The two factors on the right-hand side of \eqref{eq:middle1} are clearly bounded in $\beta'$ and $\llambda$. The factor in \eqref{eq:middle2} can be bounded by $C\max_x \mu(x)^{p-(1+2/d)}$ by Lemma~\ref{lem:core} and the factor in \eqref{eq:middle3} is bounded by $4^{-(k-1)}$ due to the choice of $\beta_0$ and $\lambda_0$ and Lemma~\ref{lem:core2}. This proves  \eqref{eq:K}.

\smallskip The argument for \eqref{eq:L} follows similarly. We start with  \eqref{eq:insert}, where the summation is instead over  $k\geq l+1$ and $0\leq t_0\leq t_1\leq\dots t_{k-1}\leq n$ such that $t_i\geq t_{i-1}+T$ ($i=1,\dots,k$), and $t_i-t_{i-1}\geq n^{1/2}$ for at least $l$ indices $i\in\{1,\dots,k-1\}$. By repeating the argument, we see that the left-hand side of \eqref{eq:L} is bounded by the same quantity  as in \eqref{eq:middle1}--\eqref{eq:middle3}, except that \eqref{eq:middle3} is replaced by
\begin{align*}
	&\sum_{k\geq l+1}\binom{k-1}{l}\Big(\sum_{t\geq T,z\in\Z^d} \E\left[W^{\otimes 2,\beta',\llambda}_t\big[\1_{X_s\neq X_s'\text{ for all }s=T,\dots,t-1,X_t=X_t'=z}\big]^{p/2}\right]\Big)^{k-1-l}\\
	&\qquad\times \Big(\sum_{t\geq n^{1/2},z\in\Z^d} \E\left[W^{\otimes2,\beta',\llambda}_t\big[\1_{X_s\neq X_s'\text{ for all }s=T,\dots,t-1,X_t=X_t'=z}\big]^{p/2}\right]\Big)^{l}.
\end{align*}
The binomial coefficient is bounded by $2^{k-1}$, the term inside $(\dots)^{k-1-l}$ is bounded by $\frac{1}{4}$ due to \eqref{eq:first} and the term inside $(\dots)^l$ by $Cn^{-(p-(1+2/d))/2}$ due to \eqref{eq:second}. Thus each summand in the display above is bounded by $2^{-(k-1-l)}(2Cn^{-(p-(1+2/d))/2})^l$ and the claim follows by taking the sum over $k$.

\smallskip Finally, to obtain \eqref{eq:KK} we use \eqref{eq:sa} to get
\begin{align*}
	\E[(W^{\beta',\llambda})^p]\leq \sum_{k=0}^\infty\E\big[W^{\otimes 2,\beta',\llambda}[\1_{K_n=k}]^{p/2}\big].
\end{align*}
Applying \eqref{eq:K}, we see that the contribution from the sum over $k\geq 1$ is indeed bounded in $n$, $\beta'$ and $\llambda$, and by Jensen's inequality $\E[W_n^{\otimes 2,\beta',\llambda}[\1_{K_n=0}]^{p/2}]\leq \E[W_n^{\otimes 2,\beta',\llambda}[\1_{K_n=0}]]^{p/2}\leq 1$.
\end{proof}

Up to this point, we have only considered partition functions with free endpoint. We now bound the joint pinned partition function with given starting point and endpoint and with the restriction that there is at least one collision.

\begin{lemma}\label{lem:core3}
	Assume $\p(\beta)>1+2/d$ and let $p\in(1+2/d,\p(\beta)\wedge 2)$. Recalling \eqref{eq:def_EE} and \eqref{eq:def_m}, let $p$, $T$ and $\lambda_0$ be as in Lemma~\ref{lem:chaos} and set $\alpha\coloneqq |\m(\lambda_0,\dots,\lambda_0)|_\infty$/2. For all $M>1$, there exists $C>0$ such that, for all $n\in\N$ and $\nu\in\EE_n(\alpha ,M)$,
\begin{align}\label{eq:RHS}
	\E\Big[W_{0,n}^{\otimes 2,\beta, \nu}[\1_{K_{n-1}>0}]^{p/2}\Big]\leq C\left(\sup_x \nu(x,\star)^{p-(1+2/d)}+\sup_y\nu(\star,y)^{p-(1+2/d)}\right).
\end{align}
\end{lemma}

\begin{proof}
Given $\nu\in\EE_n(\alpha ,M)$, let $a,b$ be as in \eqref{eq:def_EE} and note that there exists $\llambda\in[-\lambda_0,\lambda_0]^d$ such that $\m(\llambda)=(b-a)/n$. We decompose
\begin{equation}\label{eq:strt}
	\begin{split}	\E\left[W^{\otimes2,\beta,\nu}_{0,n}[\1_{K_{n-1}>0}]^{p/2}\right] &\leq\E\left[W^{\otimes2,\beta,\nu}_{0,n}[\1_{K_{n-1}>K\log n}]^{p/2}\right] +\E\big[W^{\otimes2,\beta,\nu}_{0,n}[\1_{L_{n-1}> L}]^{p/2}\big]\\
				   &\quad+\E\Big[W^{\otimes2,\beta,\nu}_{0,n}[\1_{0<K_{n-1}\leq K\log n,L_{n-1}\leq L}]^{p/2}\Big],
	\end{split}
	\end{equation}
We first show that $K$ and $L$ can be chosen in such a way that the first two terms in the above display can be disregarded, namely
\begin{align}
	\E\left[W^{\otimes2,\beta,\nu}_{0,n}[\1_{K_n>K\log n}]^{p/2}\right]&\leq Cn^{-\frac d2(p-(1+2/d))-1}\label{eq:b1},\\
	\E\big[W^{\otimes2,\beta,\nu}_{0,n}[\1_{L_n> L}]^{p/2}\big]&\leq Cn^{-\frac d2(p-(1+2/d))-1}\label{eq:b2}.
\end{align}
Note that $\nu(\cdot,\star)$ is supported on sets of cardinality at most $(2M+1)^dn^{d/2}$, hence $\sup_x\nu(x,\star)\geq (2M+1)^{-d}n^{-d/2}$  and therefore the right-hand side of \eqref{eq:RHS} is much larger than the right-hand sides of \eqref{eq:b1} and \eqref{eq:b2}. In order to prove \eqref{eq:b1} and \eqref{eq:b2}, note that for any event $A$,
\begin{align*}
W^{\otimes 2,\beta,\nu}_{0,n}[\1_A]=\sum_{x,x',y,'y}\nu(x,y)\nu(x',y') W^{\otimes 2,\beta,(x,x'),(y,y')}_{0,n}[\1_A]
\end{align*}
Each term can be estimated as follows:
\begin{align*}
	&W^{\otimes 2,\beta,(x,x'),(y,y')}_{0,n}[\1_A]\\&= \frac{W^{\otimes 2,\beta,(x,x'),\star}_{n-1}[\1_A\1_{X_n=y,X_n'=y'}]}{P(X_n=y-x)P(X_n=y'-x')}\\
					     &=\frac{W^{\otimes 2,\beta,(x,x'),\star}_{n-1}[\1_A\1_{X_n=y,X_n'=y'}e^{\llambda\cdot (X_n-x) +\llambda \cdot (X_n'-x')-2n\varphi(\llambda)}  ]}{E[e^{\llambda\cdot X_n -n\varphi(\llambda)}\1_{X_n=y-x}]E[e^{\llambda\cdot X_n -n\varphi(\llambda)}\1_{X_n=y'-x'}]}\\
					     &\leq \frac{W^{\otimes 2,\beta,\llambda,(x,x'),\star}_{n-1}[\1_A  ]}{P^{\llambda}(X_n=y-x)P^{\llambda}(X_n=y'-x')}\\
					     &\leq CW^{\otimes 2,\beta,\llambda,(x,x'),\star}_{n-1}[\1_A  ]n^{d},
\end{align*}
where $(x,y)$ and $(x',y')$ are taken from the support of $\nu$. In the final line, we have used that $y-x,y'-x'\in n\m(\llambda)+[-Mn^{1/2},Mn^{1/2}]^d$ and applied Theorem~\ref{thmx:local}. Together with \eqref{eq:sa}, we get
\begin{align}
	&\E\big[W^{\otimes 2,\beta,\nu}_{n-1}[\1_A]^{p/2}\big]\notag\\
	&\leq Cn^{dp/2}\sum_{x,x',y,y'} \nu(x,y)^{p/2}\nu(x',y')^{p/2} \E\left[W^{\otimes 2,\beta,\llambda,(x,x'),\star}_{n-1}[\1_A]^{p/2}\right]\notag\\
							   &\leq Cn^{pd/2}\Big(\sup_{x,x'}\E\left[W^{\otimes 2,\beta,\llambda,(x,x'),\star}_{n-1}[\1_A]^{p/2}\right]\Big)\sum_{x,x',y,y'} \nu(x,y)^{p/2}\nu(x',y')^{p/2} \notag\\
							   &\leq Cn^{pd/2+2d}\sup_{x,x'}\E\left[W^{\otimes 2,\beta,\llambda,(x,x'),\star}_{n-1}[\1_A]^{p/2}\right] .\label{eq:apriori}
\end{align}
In the final line, we have bounded $\nu(x,y)$ and $\nu(x',y')$ by $1$ and used that the support of $\nu$ has cardinality at most $(2M+1)^{2d}n^d$. Now, to prove \eqref{eq:b1} and \eqref{eq:b2} we set $A$ equal to $\{K_{n-1}\geq K\log n\}$ and to $\{L_{n-1}\geq L\}$. Due to  Lemma~\ref{lem:chaos}, we can choose $K$ and $L$ such that the expectation in \eqref{eq:apriori} decays at an arbitrarily fast polynomial rate.

\smallskip It remains to bound the final term in \eqref{eq:strt}. To do so, we partition the interval $[\frac 14n,\frac 34n]\cap\N$ into $4L$ disjoint intervals $I_1=\{s_1,\dots,t_1\},\dots,I_{4L}=\{s_{4L},\dots,t_{4L}\}$ of (approximately) equal size, such that $[\frac 14n,\frac 34n]\cap\N=\bigcup_{i=1}^{4L}I_i$ and $|I_i|\geq n/9L$ for $n$ large enough. We claim that on $\{0<K_{n-1}\leq K\log n,L_{n-1}\leq L\}$ there is at least one interval $I_i$ such that $X_t\neq X_t'$ for all $t\in I_i$. To see this, let $A_0\coloneqq 0$ and let $A_{i+1}$ be the index $j$ such that $\tau_j$ is the endpoint of the next large interval after time $\tau_{A_i}$,
\begin{align*}
	A_{i+1}\coloneqq \inf\big\{j>A_i\colon \tau_{j}\geq \tau_{j-1}+ n^{1/2}\big\}.
\end{align*}
After each large interval $[\tau_{A_i-1},\tau_{A_i}]$ there are $A_{i+1}-A_i-1\leq K_n$ intervals of length at most $n^{1/2}$, thus $\{\tau_{A_i+1},\dots,\tau_{A_{i+1}-1}\}\subseteq [\tau_{A_i},\tau_{A_i}+ K_nn^{1/2}]$.
On $\{K_{n-1}\leq K\log n\}$, it holds that
\begin{align*}
	\big\{t\in\{1,\dots,n-1\}\colon X_t=X_t'\big\}\subseteq\bigcup_{j=0}^{K_{n-1}-1}[\tau_{j},\tau_{j}+T)\subseteq\bigcup_{i=0}^{L_{n-1}}\big[\tau_{A_i},\tau_{A_i}+Kn^{1/2}\log n+T\big).
\end{align*}
The length of the intervals in the last union is much smaller than $|I_i|$ (recall that $|I_i|\geq n/9L$ for $n$ large enough), so each interval can intersect at most $2$ of the $I_i$'s. Hence on $\{L_{n-1}\leq L\}$ there must be at least one $i_0$ such that $\bigcup_{i=0}^{K_{n-1}-1}[\tau_i,\tau_i+T)\cap I_{i_0}=\emptyset$.

\smallskip In addition, on $K_{n-1}>0$ we must have $X_r=X_r'$ for some $r$ in either $[0,s_{i_0})\cap\N$ or $(t_{i_0},n)\cap\N$. Together with \eqref{eq:sa}, we see that the last term in \eqref{eq:strt} is bounded by
	\begin{align*}
		&\sum_{i=1}^{4L}\E\Big[W^{\otimes2,\beta,\nu}_{0,n}[\1_{X_r=X_r'\text{ for some }r<s_i,X_t\neq X_t'\text{ for all }t\in I_i}]^{p/2}\Big]\\
		&\qquad+\sum_{i=1}^{4L}\E\Big[W^{\otimes2,\beta,\nu}_{0,n}[\1_{X_r=X_r'\text{ for some }r>t_i,X_t\neq X_t'\text{ for all }t\in I_i}]^{p/2}\Big].
\end{align*}
We will bound the first sum and note that the contribution from the second sum can be treated similarly. As explained in Section~\ref{sec:strat}, the idea is that since there are no intersections in $I_i$, we can integrate out the environment in this strip, and the resulting partition function is then essentially the product of two free partition functions. Indeed, by Jensen's inequality, each summand can be bounded by
\begin{align}
		&\E\Big[\E\big[W^{\otimes2,\beta,\nu}_{0,n}[\1_{X_r=X_r'\text{ for some }r<s_i,X_t\neq X_t'\text{ for all }t\in I_i}]\big|\F_{(0,n)\setminus I_i}\big]^{p/2}\Big]\notag\\
		&= \E\big[W^{\otimes2,\beta,\nu}_{(0,n)\setminus I_i}[\1_{X_r=X_r'\text{ for some }r<s_i,X_t\neq X_t'\text{ for all }t\in I_i}]^{p/2}\big]\notag\\
		&\leq  \E\big[W^{\otimes2,\beta,\nu}_{(0,n)\setminus I_i}[\1_{X_r=X_r'\text{ for some }r<s_i}]^{p/2}\big]\notag
	\end{align}
	where, for $I\subseteq \R_+$, $\F_{I}\coloneqq \sigma(\omega_{t,x}\colon t\in I\cap \N)$ and where (recall \eqref{eq:HI})
	\begin{align*}
		&W^{\otimes2,\beta,\nu}_{I}[\1_{A}]\\
		&\sum_{x,x',y,y'}\nu(x,y)\nu(x',y')E^{\otimes 2,(x,x'),(y,y')}_{0,n}\big[e^{\beta H_{I}(\omega,X)+\beta H_{I}(\omega,X')-2|I\cap \N|\lambda(\beta)}\1_A\big].
	\end{align*}
	Now, applying Lemma~\ref{lem:lclt} with
	\begin{align*}
		f&=e^{\beta H_{[1,s_i)}(\omega,X)+\beta H_{[1,s_i)}(\omega,X')-2s_i\lambda(\beta)}\1_{X_r=X_r'\text{ for some }r<s_i},\\
		g&=e^{\beta H_{(t_i,n)}(\omega,X)+\beta H_{(t_i,n)}(\omega,X')-2(n-t_i)\lambda(\beta)}
	\end{align*}
shows  that, almost surely,
\begin{align*}
	W^{\otimes2,\beta,\nu}_{(0,n)\setminus I_i}[\1_{X_r=X_r'\text{ for some }r<s_i}]&\leq CW^{\otimes 2,\beta,\llambda,\nu(\cdot,\star),\star}_{0,s_i-1}[\1_{X_r=X_r'\text{ for some }r<s_i}]W^{\otimes 2,\beta,\llambda,\star,\nu(\star,\cdot)}_{t_i+1,n}\\
											&=CW^{\otimes 2,\beta,\llambda,\nu(\cdot,\star),\star}_{0,s_i-1}[\1_{K_{s_i-1}>0}]W^{\otimes 2,\beta,\llambda,\star,\nu(\star,\cdot)}_{t_i+1,n}.
\end{align*}
Using again \eqref{eq:sa}, we have
\begin{align*}
	&\E\Big[W^{\otimes2,\beta,\nu}_{0,n}[\1_{X_r=X_r\text{ for some }r<s_i,X_t\neq X_t\text{ for all }t\in I_i}]^{p/2}\Big]\\
	&\leq \E\Big[(W^{\beta,\llambda,\star,\nu(\star,\cdot)}_{t_i+1,n}\big)^{p}\Big]\sum_{k\geq 1}\E\Big[W^{\otimes 2,\beta,\llambda,\nu(\cdot,\star),\star}_{0,s_i-1}[\1_{K_{s_i-1}=k}]^{p/2}\Big]
\end{align*}
The conclusion thus follows from Lemma~\ref{lem:chaos} by using \eqref{eq:KK} for the left factor and summing \eqref{eq:K} over $k\geq 1$ for the right factor.
	\end{proof}

\section{Proof of Theorem~\ref{thm:local}}\label{sec:main}
\begin{proof}[Proof of Theorem~\ref{thm:local}]
Let $T$ and $\alpha $ be as in Lemma~\ref{lem:core3}. Using \eqref{eq:sa} and Jensen's inequality, we decompose
	\begin{align*}
\E\big[\big(W^{\beta,x,y}_{0,n}\big)^p\big]&=\E\big[\big(W^{\otimes2,\beta,(x,x),(y,y)}_{0,n}\big)^{p/2}\big]\\
					   &\leq 1 +\E\Big[W^{\otimes2,\beta,(x,x),(y,y)}_{0,n}[\1_{K_{n-1}>0}\big]^{p/2}\Big],
	\end{align*}
The second term is bounded due to \eqref{eq:RHS} (with $\nu$ the Dirac distribution on $(x,y)$). It remains to prove \eqref{eq:local1}, for which we first observe
\begin{align}\label{eq:goal}
		\E[|W^{\beta,\nu}_{0,n}-1|^p]\leq \E\big[|W^{\beta,\nu}_{0,n}-1|^p\1_{A^c}\big]+\E\big[|W^{\beta,\nu}_{0,n}-1|^p\1_A\big],
	\end{align}
	where, for $a\in(0, 1)$ and $r\coloneqq \max_x\nu(x,\star)+\max_y\nu(\star,y)$,
	\begin{align*}
		A&\coloneqq \big\{W^{\otimes 2,\beta,\nu}_{0,n}[\1_{K_{n-1}>0}]\leq r^a\big\}.
\end{align*}
Note that, by choosing $p'$ sufficiently close to $\p$ we obtain, using again Lemma~\ref{lem:core3},
\begin{align*}
	\P(A^c)&\leq r^{-ap'/2}\E\big[W_{0,n}^{\otimes 2,\beta,\nu}[\1_{K_{n-1}>0}]^{p'/2}\big]\\
	       &\leq Cr^{-ap'/2+p'-(1+2/d)}\\
	       &\leq Cr^{-a\p/2+\p-(1+2/d)-\eps/2}.
\end{align*}
Then, by choosing $p''$ sufficiently close to $\frac{\p}{p}$, we obtain
\begin{equation}\label{eq:first_term}\begin{split}
	\E\big[|1-W_{0,n}^{\beta,\nu}|^p\1_{A^c}\big]&\leq \E\big[(1+W_{0,n}^{\beta,\nu})^{pp''}\big]^{\frac{1}{p''}}\P(A^c)^{1-\frac{1}{p''}}\\
						     &\leq \sup_{x,y\colon \nu(\{x,y\})>0}\E\big[(1+W_{0,n}^{\beta,x,y})^{pp''}\big]^{\frac{1}{p''}}\P(A^c)^{1-\frac{1}{p''}}\\
						     &\leq Cr^{(\p-p)(1-\frac{1+2/d}\p-a/2)-\eps}.
\end{split}\end{equation}
The second inequality is due to Jensen's inequality and the third inequality uses \eqref{eq:local3}. We have obtained a bound for the first term in \eqref{eq:goal}.

\smallskip Similarly, by choosing $p'''$ sufficiently close to $\p$, we get
\begin{equation}\label{eq:second_term1}
\begin{split}
	\E\big[W_{0,n}^{\beta,\nu}\1_{A^c}\big]&\leq \sup_{x,y\colon \nu(\{x,y\})>0}\E\big[(W^{\beta,x,y}_{0,n})^{p'''}\big]^{\frac{1}{p'''}}\P(A^c)^{1-\frac{1}{p'''}}\\
					       &\leq Cr^{(\p-1)(1-\frac{1+2/d}\p-a/2)-2\eps/p}.
\end{split}
\end{equation}
To bound the second term in \eqref{eq:goal}, we use
\begin{align*}
	&\E[|W^{\beta,\nu}_{0,n}-1|^p\1_A]\\
					 &\leq \E[|W^{\beta,\nu}_{0,n}-1|^2\1_A]^{p/2}\\
					 &=\left(\E[W^{\otimes 2,\beta,\nu}_{0,n}\1_A]-2\E[W^{\beta,\nu}_{0,n}\1_A]+\P(A)\right)^{p/2}\\
					 &=\left(\E[W^{\otimes 2,\beta,\nu}_{0,n}[\1_{K_{n-1}>0}]\1_A]+\E[W^{\otimes 2,\beta,\nu}_{0,n}[\1_{K_{n-1}=0}]\1_A]+2\E[W^{\beta,\nu}_{0,n}\1_{A^c}]-\P(A^c)-1\right)^{p/2}\\
					 &\leq \left(\E[W^{\otimes 2,\beta,\nu}_{0,n}[\1_{K_{n-1}>0}]\1_A]+2\E[W^{\beta,\nu}_{0,n}\1_{A^c}]-P(A^c)\right)^{p/2}\\
					 &\leq \left(r^{a}+2Cr^{(\p-1)(1-\frac{1+2/d}\p-a/2)-2\eps/p}\right)^{p/2}\\
					 &\leq r^{a p/2}+C'r^{\frac p2(\p-1)(1-\frac{1+2/d}\p-a/2)-\eps}
\end{align*}
In the second equality, we have used that $\E[W^{\beta,\nu}_{0,n}\1_A]=1-\E[W^{\beta,\nu}_{0,n}\1_{A^c}]$ and in the second inequality we have used $\E[W^{\otimes 2,\beta,\nu}_{0,n}[\1_{K_{n-1}=0}]]\leq 1$. The third equality uses the definition of $A$ and \eqref{eq:second_term1}. Together with \eqref{eq:first_term}, \eqref{eq:goal} is now bounded by
\begin{align*}
	r^{a p/2} + r^{\frac p2 (\p-1)(1-\frac{1+2/d}\p-a/2)-\eps} + r^{(\p-p)(1-\frac{1+2/d}\p-a/2)-\eps}.
\end{align*}
We obtain \eqref{eq:local1} by optimizing this expression over $a\in(0,2(1-\frac{1+2/d}\p))$. The first exponent is increasing and the remaining two are decreasing, and moreover they are multiples of each other. In particular, they equal zero for the same value of $a$. Depending on whether $p>1+\frac{\p-1}{\p+1}$ or not, the minimizing value $a $ can be computed by setting the first and the third exponent, resp. the first and the second exponent to be equal, which gives
\begin{align*}
	a =2\Big(1-\frac{1+2/d}{\p}\Big)\cdot\begin{cases}\frac{\p-1}{\p+1}&\text{ if }p\in(1+\frac2d,1+\frac{\p-1}{\p+1}),\\
	\frac{\p-p}{\p}&\text{ if }p\in(1+\frac{\p-1}{\p+1},\p).\end{cases}
\end{align*}
We obtain $\xi$ as $a p/2$, which concludes the proof.
\end{proof}

\section{Proof of the properties of $\p$}\label{sec:p}

\begin{proof}[Proof of Theorem~\ref{thm:p}]
	It is well-known that $\beta\mapsto \E[f(W_n^\beta)]$ is weakly increasing for every $f\colon\R_+\to\R$ convex, see for example \cite[Remark 2.6]{C17}, so $\beta\mapsto\p(\beta)$ is weakly decreasing.

	\smallskip For \textbf{part  (i)} we assume $\p(\beta)\in(1+2/d,2]$. By Lemma~\ref{lem:chaos}, for every $p\in(1+2/d,\p)$ there exists $\beta'>\beta$ such that $\sup_n\E[(W_n^{\beta'})^p]<\infty$, hence $\p(\beta')\geq p$. By taking $p\uparrow\p$, we obtain $\lim_{\beta'\downarrow\beta}\p(\beta')\geq\p$ and the claim follows from the monotonicity of $\p$.

	\smallskip For \textbf{part (ii)}, we use hypercontractivity and the result from \cite{J21_2}, which are introduced in Appendix~\ref{app:hyper} and~\ref{app:21_2}. We assume that $\P$ has finite support and that $\p(\beta)>1$. Let $\beta'\in(0,\beta)$. By Theorem~\ref{thmx:alea}, there exists $\rho\in(0,1)$ such that, for any $p\geq 1$ and $n\in\N$,
	\begin{align}\label{eq:compare}
		\E[(W_n^{\beta'})^p]\leq \E[(T_\rho W_n^\beta)^p].
	\end{align}
Next, recall the value $p(q)$ from Theorem~\ref{thmx:myhyper}. We claim that there exists $\eps>0$ such that
	\begin{align}\label{eq:eps}
		p(\p(\beta)-\eps)>\p(\beta).
	\end{align}
	Indeed, for all $\eps\in(0,\p(\beta)-1)$ it holds that $r(\eps)\coloneqq \frac{p(\p(\beta)-\eps)}{\p(\beta)-\eps}>1$ and $\eps\mapsto r(\eps)$ is increasing. Let $\eps_0\coloneqq \frac{1+\p(\beta)}{2}$ and note that \eqref{eq:eps} holds if we choose $\eps>0$ small enough for $(\p(\beta)-\eps)r(\eps_0)>\p(\beta)$. Finally, we apply \eqref{eq:compare} with $p$ equal to $p(\p(\beta)-\eps)$ and obtain
	\begin{align*}
		\|W_n^{\beta'}\|_{p(\p(\beta)-\eps)}\leq\|T_\rho W_n^{\beta}\|_{p(\p(\beta)-\eps)}\leq \|W_n^\beta\|_{\p(\beta)-\eps},
	\end{align*}
	where the last inequality is due to Theorem~\ref{thmx:myhyper}. By definition of $\p$, the right-hand side is bounded in $n$, hence by \eqref{eq:eps} we have $\p(\beta')\geq p(\p(\beta)-\eps)>\p(\beta)$.
\end{proof}

\section{Proofs for the corollaries}\label{sec:coro}

We start with the local limit theorem for the polymer measure.

\begin{proof}[Proof of Corollary~\ref{cor:informal}]
For \eqref{eq:informal2}, we introduce
\begin{align*}
	W^{\beta,x,y}_{(0,r]\cup[n-r,n)}&\coloneqq E^{x,y}_{0,n}[e^{\beta H_{(0,r]\cup [n-r,n)}(\omega,X)-2r\lambda(\beta)}],\\
	\nu_{\omega,r}(x',y')&\coloneqq \frac{W^{\beta,x,y}_{(0,r]\cup[n-r,n)}[\1_{X_r=x',X_{n-r}=y'}]}{W^{\beta,x,y}_{(0,r]\cup[n-r,n)}}.
\end{align*}
By applying Lemma~\ref{lem:lclt}(ii) with $f=e^{\beta H_{(0,r]}(\omega,X)-r\lambda(\beta)}$ and $g=e^{\beta H_{[n-r,n)}(\omega,X)-(n-r)\lambda(\beta)}$, we obtain, almost surely,
\begin{align}\label{eq:uniform1}
	\Big|\frac{W^{\beta,x,y}_{(0,r]\cup[n-r,n)}}{W^{\beta,x,\star}_{r}W^{\beta,\star,y}_{n-r,n}}-1\Big|\leq a_n.
\end{align}
Moreover, for $n$ large enough the reciprocal of the quotient is within distance $2a_n$ of $1$. Similarly,  by multiplying $f$ by $\1_{X_r=x'}$ and $g$ by and $\1_{X_{n-r}=y'}$, for $(0,x)\leftrightarrow(r,x')$ and $(n-r,y')\leftrightarrow(n,y)$,
\begin{align*}
	\Big|\frac{W^{\beta,x,y}_{(0,r]\cup[n-r,n)}[\1_{X_r=x'}\1_{X_{n-r}=y'}]}{W^{\beta,x,\star}_{r}[\1_{X_r=x'}]W^{\beta,\star,y}_{n-r,n}[\1_{X_{n-r}=y'}]}-1\Big|\leq a_n.
\end{align*}
Thus we obtain, almost surely for $n$ large enough,
\begin{align}\label{eq:uniform2}
	\sup_{x',y'}\Big|\frac{\nu_{\omega,r}(x',y')}{\mu_{\omega,r}^{x,\star}(X_r=x')\mu_{\omega,[n-r,n)}^{\star,y}(X_{n-r}=y')}-1\Big|\leq 4a_n,
\end{align}
where $\mu_{\omega,r}^{x,\star}(X_{r}=x')=\frac{W^{\beta,x,\star}_{r}[\1_{X_{r}=x'}]}{W^{\beta,x,\star}_{r}}$ and $\mu_{\omega,[n-r,n)}^{\star,y}(X_{n-r}=y')=\frac{W^{\beta,\star,y}_{n-r,n}[\1_{X_{n-r}=y'}]}{W^{\beta,\star,y}_{n-r,n}}$ are the forward and backward versions of the polymer measure and the supremum is over those $x',y'$ with $\nu_{\omega,r}(x',y')>0$.

\smallskip Note that $W^{\beta,x,y}_{0,n}=W^{\beta,x,y}_{(0,r]\cup[n-r,n)}W^{\beta,\nu_{\omega,r}}_{r,n-r}$. Thus, using \eqref{eq:uniform1} and the inequality $|AB-1|\leq |A||B-1|+|A-1|$, we obtain
\begin{align*}
	\Big|\frac{W^{\beta,x,y}_{0,n}}{W^{\beta,x,\star}_{r}W^{\beta,\star,y}_{n-r,n}}-1\Big|\leq(1+a_n) |W^{\beta,\nu_{\omega,r}}_{r,n-r}-1|+ a_n.
\end{align*}
Furthermore, by applying Theorem~\ref{thm:local}(ii) and \eqref{eq:uniform2}, we get, for some $A>0$,
\begin{align*}
	\E[|W^{\beta,\nu_{\omega,r}}_{r,n-r}-1|^p]&\leq C \E[\max_{x'} \nu_{\omega,r}(x',\star)^A]+C\E[\max_{y'}\nu_{\omega,r}(\star,y')^A]\\
						  &\leq C(1+4a_n)^A \E\big[\max_{x'} \mu_{\omega,r}(X_r=x')^A\big].
\end{align*}
Note that the bounded in the last line no longer depends on $x$ and $y$. The claim now follows from Theorem~\ref{thmx:B}(i). For \eqref{eq:informal3}, due to the preceding considerations, we obtain
\begin{align*}
&\E\Big[\Big|W^{\beta,x,y}_{0,n}-W^{\beta,x,\star}_rW^{\beta,\star,y}_{n-r,n}\Big|^p\Big]=\E\Big[\Big(W^{\beta,x,\star}_rW^{\beta,\star,y}_{n-r,n}\Big)^p\Big|\frac{W^{\beta,x,y}_{0,n}}{W^{\beta,x,\star}_rW^{\beta,\star,y}_{n-r,n}}-1\Big|^p\Big]\\
&\leq C\E\Big[\Big(W^{\beta,x,\star}_rW^{\beta,\star,y}_{n-r,n}\Big)^p\Big(\max_{x'} \mu_{\omega,r}^{x,\star}(X_r=x')^A+\max_{y'}\mu_{\omega,[n-r,n)}^{\star,y}(X_{n-r}=y')^A+ a_n\Big)\Big]
\end{align*}
The second factor in the last expectation converges to zero in probability by Theorem~\ref{thmx:B}(i), and since it is bounded also in $L^q$ for any $q<\infty$. Thus the claim follows by H\"older's inequality.
\end{proof}

We continue with the critical exponent for the pinned partition function.
\begin{proof}[Proof of Corollary~\ref{cor:pp}]
By applying Lemma~\ref{lem:lclt}(ii) with $\eps=\frac 18$, $x=y=0$, $s=\lfloor 2n^{\eps}\rfloor$, $t=0$, $f=e^{\beta H_s(\omega,X)-s\lambda(\beta)}$ and $g=1$, we obtain
\begin{align*}
W_s^{\beta,0,\star}\leq \frac{\E[W_{0,2n}^{\beta,0,0}|\F_s]}{1-a_{2n}},
\end{align*}
and thus, by Jensen's inequality,
\begin{align*}
\E[(W_s^\beta)^p]\leq \E[(W_{0,2n}^{\beta,0,0})^p](1-a_{2n})^{-p}.
\end{align*}
If $p> \p(\beta)$, then the left-hand side is unbounded in $n$, hence the right-hand side is also unbounded and we obtain $p\geq \pp(\beta)$. This shows $\pp(\beta)\leq\p(\beta)$.

\smallskip For the converse inequality we assume $\p(\beta)\in(1+2/d,2]$. By Theorem~\ref{thm:local}(i), for any $p\in(1+2/d,\p)$ we have $\sup_n\E[(W_{0,2n}^{\beta, 0,0})^p]<\infty$, hence $\pp(\beta)\geq p$. By taking $p\uparrow\p$ we obtain $\pp\geq\p$.
\end{proof}

Next, we prove the stability of $W^{\beta,\llambda}_n$ in the drift.
\begin{proof}[Proof of Corollary~\ref{cor:drift}]
	We assume $\p(\beta)\in(1+2/d,2]$. For any $p\in(1+2/d,\p)$, by Lemma~\ref{lem:core3} there exists $\lambda_0$ such that $\sup_n\E[(W_n^{\beta,\llambda})^p]<\infty$ holds for all $\llambda\in[-\lambda_0,\lambda_0]^d$, hence $\p(\beta,\llambda)\geq p$. The lower-semicontinuity follows by taking $p\uparrow\p$. For the second claim, note that $(W^{\beta,\llambda}_n)_{n\in\N}$ is again a martingale satisfying a zero-one law. Thus \\$\sup_n\E[(W_n^{\beta,\llambda})^p]<\infty$ for $p>1$ implies $W_\infty^{\beta,\llambda}>0$.
\end{proof}

Next, we prove the claim regarding the large deviation principle.

\begin{proof}[Proof of Corollary~\ref{cor:LDP}]
	The existence of the LDP is well-known, even in the strong disorder phase, see for example \cite[Chapter 9.2]{C17}. We compute the logarithmic moment generating function of $\mu_{\omega,n}^\beta(X_n\in\cdot)$,
	\begin{align*}
		\log \mu_{\omega_n}^{\beta}[e^{\llambda\cdot X_n}]=\log E[e^{\beta H_n(\omega,X)-n\lambda(\beta)+\lambda X_n}]-\log W^{\beta}_n=\log W^{\beta,\llambda}_n+n\varphi(\llambda)-\log W^{\beta}_n.
	\end{align*}
	By standard arguments based on concentration inequalities, for any $\llambda\in\R^d$, almost surely,
	\begin{align*}
		\lim_{n\to\infty}\frac{1}{n}\log W_n^{\beta,\llambda}=\lim_{n\to\infty}\frac{1}{n}\E[\log W_n^{\beta,\llambda}]\leq 0,
	\end{align*}
	where the last inequality is Jensen's inequality. If \eqref{eq:WD} holds then $W_n^\beta$ converges almost surely to a positive limit, hence
	\begin{align}
		\lim_{n\to\infty}\frac{1}{n}\log \mu_{\omega,n}^\beta[e^{\llambda\cdot X_n}]\leq \varphi(\llambda).\label{eq:dfsfsd}
	\end{align}
	The inequality $I^\beta(x)\geq I^0(x)$ follows from the G\"artner-Ellis Theorem, see for example \cite[Theorem~2.3.7]{DZ}.

	\smallskip Moreover, if \eqref{eq:WD} holds with bias $\llambda$, then $\lim_{n\to\infty}\frac{1}{n}\log W_n^{\beta,\llambda}=0$ almost surely and the inequality in \eqref{eq:dfsfsd} becomes an equality. Thus for $\p(\beta)\in(1+2/d,2]$ the equality of $I^\beta$ and $I^0$ in a neighborhood of the origin follows from \eqref{eq:weak_stable} and the G\"artner-Ellis Theorem~(note that the behavior of $I^\beta$ around zero is determined by the behavior of $\llambda\mapsto \lim_{n\to\infty}\frac 1n\log\mu_{\omega,n}^\beta[e^{\llambda\cdot X_n}]$ for small $\llambda$).

\smallskip 	On the other hand, if $\p(\beta)>2$ then the equality of $I^\beta$ and $I^0$ in a neighborhood of the origin has been observed in \cite[Exercise 9.1]{C17}, and for $\p(\beta)\in(1+2/d,2]$ we can apply Corollary~\ref{cor:drift} to conclude.
\end{proof}

Next, we prove the statements regarding the decay of the replica overlap.

\begin{proof}[Proof of Corollary~\ref{cor:I}]
	\textbf{Part (i)}: Note that since $I_n^{\beta,p}$ is decreasing in $p$, it is enough to check the claim for $p$ close to $1+2/d$. Assume $\p(\beta)>1+2/d$ and let $p\in(1+2/d,\p\wedge 2)$. Let $\alpha$ be as in Theorem~\ref{thm:local}. By standard large deviation estimates for the simple random walk, see for example \cite[Theorem~2.2.3(a)]{DZ}, there exists $c>0$ such that, for all $n\in\N$ and $|x|\geq \alpha n$,
	\begin{align}\label{eq:asdasdasdasd}
		\P\big(W^\beta_n[\1_{X_{n+1}=x}]\geq e^{-cn}\big)\leq e^{cn}P(X_{n+1}=x)\leq e^{-cn},
	\end{align}
	and hence, almost surely,
	\begin{align}\label{eq:distant}
		W^\beta_n[\1_{X_{n+1}=x}]\1_{|x|\geq \alpha n}\geq e^{-cn}\text{ for at most finitely many }n\in\N,x\in\Z^d.
	\end{align}
	We decompose
	\begin{align*}
		\sum_n I_n^{\beta,p}&=(W^\beta_n)^{-p}\sum_{n\in\N,x\in\Z^d}W^\beta_n[\1_{X_{n+1}=x}]^{p}\\
				    &=(W^\beta_n)^{-p}\left(\sum_{n\in\N,|x|>\alpha n}W^\beta_n[\1_{X_{n+1}=x}]^p+\sum_{n\in\N,|x|\leq \alpha n}(W_{0,n+1}^{\beta,0,x})^pP(X_{n+1}=x)^p\right).
	\end{align*}
	Due to \eqref{eq:distant}, the first sum is almost surely finite. For the second sum, we take expectation and apply \eqref{eq:local3},
	\begin{equation}\begin{split}\label{eq:ada}
		\sum_{n\in\N,|x|\leq \alpha n}\E\big[(W_{0,n+1}^{\beta,0,x})^p\big]P(X_{n+1}=x)^p&\leq C\sum_{n\in\N,x\in\Z^d}P(X_{n+1}=x)^p\\
												 &\leq C\sum_{n\in\N}\max_{x\in\Z^d}P(X_{n+1}=x)^{p-1},
\end{split}\end{equation}
which is finite due to Theorem~\ref{thmx:local}. For \textbf{part (ii)}, we first establish a bound for $I_{k,n}^{\beta,2}$ that does not depend on $n$. Indeed,
\begin{align}
	I_{k,n}^{\beta,2}&=(W_n^\beta)^{-2}\sum_x W^\beta_n[\1_{X_k=x}]^2\notag\\
	       &\leq \big(\sup_m (W_m^\beta)^{-2}\big)  \sum_x W^\beta_k[\1_{X_k=x}]^2(\sup_m W^\beta_m\circ\theta_{k,x})^2\notag\\
	       &\leq \big(\sup_m (W_m^\beta)^{-2}\big) \Big( e^{-2ck}\sum_{|x|>\alpha  k} (\sup_m W^\beta_m\circ\theta_{k,x})^2+\sum_{|x|\leq\alpha k} W^\beta_k[\1_{X_k=x}]^2(\sup_m W^\beta_m\circ\theta_{k,x})^2\Big),\label{eq:again}
\end{align}
where the last inequality holds due to \eqref{eq:distant} and is valid for all $k\geq k_0(\omega)$. It is enough to show that the last line is almost surely summable over $k$. To deal with the first sum, we estimate
\begin{align*}
	\E\Big[\Big(\sum_ke^{-2ck}\sum_{|x|\geq \alpha k} (\sup_m W^\beta_m\circ\theta_{k,x})^2\Big)^{1/2}\Big]&\leq
	\E\Big[\sum_ke^{-ck}\sum_{|x|\geq \alpha k} (\sup_m W^\beta_m\circ\theta_{k,x})\Big]\\
													       &= C\sum_ke^{-ck}k^d,
\end{align*}
where we used \eqref{eq:sa} and Theorem~\ref{thmx:p}(i). For the second sum in \eqref{eq:again}, we take $p\in(1+2/d,\p\wedge 2)$ and estimate
\begin{align*}
	&\E\Big[\Big(\sum_k\sum_{|x|\leq \alpha k } W^\beta_k[\1_{X_k=x}]^2(\sup_m W^\beta_m\circ\theta_{k,x})^2\Big)^{p/2}\Big]\\&\leq \E\Big[\sum_k\sum_{|x|\leq\alpha k} W^\beta_k[\1_{X_k=x}]^p(\sup_m W^\beta_m\circ\theta_{k,x})^{p}\Big]\\
												 &=\E\Big[\big(\sup_m W^\beta_m\big)^p\Big]\E\Big[\sum_{k\in\N,|x|\leq \alpha k}W^\beta_k[\1_{X_k=x}]^p\Big].
\end{align*}
The first factor is finite due to Doob's inequality and the second factor is almost the same as in \eqref{eq:ada}, hence finite.
\end{proof}

Finally, we prove the claim regarding the typical behavior of $I_n^{\beta,2}$.

\begin{proof}[Proof of Corollary~\ref{cor:typical}]
	Since $\max_x\mu_{\omega,n}^{\beta}(X_{n+1}=x)\leq (I_n^{\beta,2})^{1/2}$, it is enough to bound $(I_n^{\beta,2})^{1/2}$. Moreover, since $\sup_n (W^\beta_n)^{-2}$ is almost surely finite, it is enough to show that
	\begin{align*}
		\lim_{n\to\infty}	\P\Big(\sum_xW_n^\beta[\1_{X_{n+1}=x}]^2\geq n^{-d(1-\frac{1}{\p\wedge 2})+\eps}\Big)=0.
	\end{align*}
	Using \eqref{eq:asdasdasdasd} and a union bound, we obtain
	\begin{align*}
		\lim_{n\to\infty}\P\Big(\sum_{|x|\geq \alpha n}W^\beta_n[\1_{X_{n+1}=x}]^2\geq e^{-cn}\Big)=0.
	\end{align*}
	To estimate the remaining sum, we fix $p\in(1+2/d,\p(\beta)\wedge 2)$ and apply the Markov inequality,
	\begin{align}
		&\P\Big(\sum_{|x|\leq\alpha n}W^\beta_n[\1_{X_{n+1}=x}]^2\geq n^{-d(1-\frac{1}{\p\wedge 2})+\eps}\Big)\notag\\
		&\leq n^{\frac{d}2(p-\frac{p}{\p\wedge 2})-\frac{\eps p}2}\E\Big[\Big(\sum_{|x|\leq\alpha n}W^\beta_n[\1_{X_{n+1}=x}]^2\Big)^{p/2}\Big].
	\label{eq:asasdaxzz}
\end{align}
To bound the expectation, we apply \eqref{eq:sa} and argue as in \eqref{eq:ada} to obtain
\begin{align*}
	\E\Big[\Big(\sum_{|x|\leq\alpha n}W^\beta_n[\1_{X_{n+1}=x}]^2\Big)^{p/2}\Big] &\leq \sum_{|x|\leq\alpha n}\E\big[(W_{0,n+1}^{\beta,0,x})^{p}\big]P(X_{n+1}=x)^p\\
										      &\leq C\max_{x\in\Z^d}P(X_{n+1}=x)^{p-1}\\
										      &\leq Cn^{-\frac{d}{2}(p-1)}.
	\end{align*}
	By comparing with \eqref{eq:asasdaxzz}, we see that the exponent of $n$ is negative if we choose $p$ sufficiently close to $\p\wedge 2$.
\end{proof}

\section*{Appendix}
\renewcommand{\thesubsection}{A\arabic{subsection}}
\subsection{Discussion of hypercontractivity}\label{app:hyper}

We assume that $\P$ has finite support $S\subseteq\R$. Let us first make the setup more precise: let $\Omega=S^\Lambda$ denote the enlarged hypercube, where $\Lambda$ is some finite index set. In our application we will take $\Lambda=\{1,\dots,n\}\times\{-n,\dots,n\}^d$. Let $\pi$ be a probability measure with support $S$ and  $\P\coloneqq \bigotimes_{i\in\Lambda} \pi$.

\smallskip For $g\colon\Omega\to\R$ and $\rho\in[0,1]$, the \emph{noise operator} $T_\rho$ acts on $g$ by
 \begin{align*}
	 (T_{\rho}g)(\omega)&\coloneqq \E[g(\omega_\rho)|\sigma(\omega)],\quad\text{ where } (\omega_\rho)_{t,x}\coloneqq\begin{cases}\omega_{t,x}&\text{ with probability }\rho\\\omega'_{t,x}&\text{ with probability }1-\rho.\end{cases}
 \end{align*}
 and where $\omega'$ is a independent copy of $\omega$. The coordinates of $\omega_{\rho}$ are independent, so that $\omega_\rho$ has law $\P$ as well. Thus $T_\rho$ smoothes out the effect of individual coordinates of $\omega$ and we expect that $\omega\mapsto (T_\rho g)(\omega)$ is, in a sense, more well-behaved than $\omega\mapsto g(\omega)$. This effect of ``noise stability'' has been investigated quite actively recently, see \cite{O14} for an overview.

\smallskip Given $1<q<p<\infty$ and $\rho\in[0,1]$, we say that a hypercontractive inequality holds if, for all $g\colon\Omega\to\R$,
 \begin{align}\label{eq:hyper}
 \|T_\rho g\|_p\leq \|g\|_q.
 \end{align}
 Note that this inequality is always satisfied for $\rho=0$ and never for $\rho=1$ (unless  $g$ is constant). One natural question is to find, for given $p$ and $q$, the largest possible $\rho$ satisfying \eqref{eq:hyper}. For our purposes, we need a slightly weaker result:

 \begin{thmx}\label{thmx:myhyper}
Fix $\rho\in(0,1)$. For every $q>1$ there exists $p=p(q)>q$ such that \eqref{eq:hyper} holds for all $g\colon\Omega\to\R$ and for all finite $\Lambda$. Moreover, $q\mapsto \frac{p(q)}{q}$ is increasing.
\end{thmx}

The hypercontractive inequality has been studied most intensively in the case $S=\{0,1\}$ and $\pi=\frac{1}{2}\delta_0+\frac{1}{2}\delta_1$, i.e., the marginals have a  symmetric Bernoulli distribution. In that case, it is known that, for given $1<q<p<\infty$, \eqref{eq:hyper} holds for all $\rho\leq (\frac{q-1}{p-1})^{1/2}$, see \cite{B70} or the discussion in \cite[Chapter 9]{O14}. Since $(\frac{q-1}{p-1})^{1/2}$ converges to one as $p\downarrow q$, we can prove Theorem~\ref{thmx:myhyper} in the symmetric Bernoulli case in this way.

\smallskip However, it seems that a generalization to biased Bernoulli distributions is not available. Namely, in that case the optimal $\rho$ has been established if either $p$ or $q$ equals $2$, see \cite[Chapter 10]{O14} and the references therein, but we did not find suitable references for the case $q<p<2$. In that direction, the most relevant work is \cite{W07}: they obtain certain bounds on the optimal $\rho$ in \eqref{eq:hyper}, but those bounds are of asymptotic nature and do not give sufficient information in the regime $p\downarrow q$. We note that they also show \cite[Theorem~3.1]{W07} that the general case $|S|>1$ can be reduced to the case  $|\Lambda|=|S|=1$, i.e., (biased) Bernoulli distributions.

\smallskip We now explain how Theorem~\ref{thmx:myhyper} can be derived from the result in \cite{AG76}, which comes from a different context. The downside of this approach is that it does not give an explicit expression for $p=p(q,\rho)$.

\begin{proof}[Proof of Theorem~\ref{thmx:myhyper}]
Fix $\rho\in(0,1)$. In \cite{AG76}, they regard $\omega$ and $\omega_\rho$ as two steps of a Markov chain on $\Omega$ with transition semigroup $T_\rho$ and study the quantity
\begin{align}\label{eq:sp}
	s_p\coloneqq \min\left\{r\in[0,1]\colon\E\big[\big((T_\rho g)(\omega)\big)^p\big]^{1/p}\leq \E\big[g(\omega_\rho)^{rp}\big]^{\frac1{rp}}\ \forall g\colon\Omega\to\R\right\}.
\end{align}
In our case, as noted above, it holds that $\omega\isDistr\omega_\rho$ and thus the definition of $s_p$ implies that \eqref{eq:hyper} holds with $q$ equal to $ps_p$. Note that $s_\rho\geq \frac 1p$: indeed, if $r\leq \frac 1p$, then
\begin{align*}
\E\big[\big((T_\rho g)(\omega)\big)^p\big]^{1/p}&\geq \E\big[\big((T_\rho g)(\omega)\big)^{rp}\big]^{1/rp}\\
&=\E\big[\E\big[g(\omega_\rho)\big|\sigma(\omega)]^{rp}\big]^{1/rp}\geq \E[g(\omega_\rho)^{rp}]^{1/rp}.
\end{align*}
The first inequality is H\"older's inequality and the second inequality is the conditional Jensen inequality, using $rp<1$. Moreover, the first inequality is strict unless $T_\rho g$ is constant, which only happens if either $\rho=0$, $g$ is constant or $\pi$ is a Dirac measure.

\smallskip Note also that $(\omega,\omega_\rho)$ is an indecomposable Markov chain in the sense given before \cite[Theorem~1]{AG76}. Thus, by \cite[Theorems~2 and 3]{AG76}, we obtain that $s_p$ is independent of $|\Lambda|$ and that $p\mapsto s_p$ is strictly decreasing with $s_1=1$.

\smallskip We now claim that we can choose $p\coloneqq q/s_q$ to conclude. Indeed, due to the strict monotonicity, $s_q<1$ and thus  $p>q$. Moreover, since $s_p<s_q$, we can take $r$ equal to $s_q$ in \eqref{eq:sp} to get
\begin{align*}
	\E[(T_\rho g)^p]^{1/p}\leq \E[g^{ps_q}]^{1/(ps_q)}=\E[g^q]^{1/q}.
\end{align*}
The final claim is clear from $\frac{p(q)}{q}=1/s_q$.
\end{proof}

It is conceivable that the result in \cite{AG76} generalizes to certain environments without finite support. Indeed, the assumption in \cite[Theorem~1]{AG76} is that $\P(\omega_{t,x}=a,(\omega_\rho)_{t,x}=b)>0$ for all $a,b$ in the support of $\P$, and this assumption continuous to hold (in the sense of densities) if, for example, $\P$ has a Lebesgues density on a compact interval which is bounded away from $0$ and $\infty$. Nonetheless, the use of hypercontractivity always introduces some technical assumptions on the environment and thus it is desirable to find a more robust proof for Theorem~\ref{thm:p}(ii).

\subsection{Hypercontractivity for partition functions}\label{app:21_2}

To apply the hypercontractive inequality to the partition function, we recall the following result from \cite{J21_2}, which  relates $T_{\rho}W_n^{\beta}$ to $W_n^{\beta'}$ for some $\beta'<\beta$.

\begin{thmx}[{\cite[Theorem~2.6]{J21_2}}]\label{thmx:alea}
If the environment is both upper and lower bounded, i.e., there exists $K>0$ such that
\begin{align*}
	\P\big(\omega_{0,0}\in[-K,K]\big)=1,
\end{align*}
then for any $0<\beta_0<\beta$ there exists $C>0$ such that, for all $\beta'\in[\beta_0,\beta]$, $n\in\N$ and $p\geq 1$,
	\begin{align*}
		\E\big[(W^{\beta'}_n)^p\big]\leq \E\Big[\big(T_{1-C(1-\frac{\beta'}{\beta})} W^\beta_n\big)^p\Big].
	\end{align*}
\end{thmx}
 \subsection{Local limit theorem for biased random walk}\label{app:local}
 In this section, we consider the local limit theorem for the biased random walk $P^\llambda$. We refer to \cite[Chapter 2]{LL10} for a discussion of this well-known result. For our purposes, we need the following statement:
	\begin{thmx}\label{thmx:local}
		For every $M>1$, there exists $C>0$ such that, for all $\llambda\in [-1,1]^d$ and $n\in\N$,
	\begin{align}
		\sup_{x\in\Z^d}P^{\llambda}(X_n=x)&\leq Cn^{-d/2}\label{eq:local_max},\\
		\inf_{\substack{x\in n\m(\llambda)+[-Mn^{1/2},Mn^{1/2}]^d\\(n,x)\leftrightarrow (0,0)}}P^{\llambda}(X_n=x)&\geq C^{-1}n^{-d/2}.\label{eq:LCLT_lower}
	\end{align}
\end{thmx}
For a fixed value of $\llambda$, it is not difficult to find references for Theorem~\ref{thmx:local}. Most works require $(X,P^\llambda)$ to be centered, aperiodic and  $\Z^d$-valued, but a version that applies to our setup can be found in \cite[Theorem~22.1]{BR10}. Most references furthermore do not provide the uniformity in $\llambda$ that we desire, although a comment in that direction is given in \cite[second comment on p. 36]{LL10}. However, upon inspection it becomes clear that the proofs can be made uniform in $\llambda$, and we will now explain how that can be checked with the help of estimates provided in \cite{LL10}.
\begin{proof}
We write $\psi_Z^\llambda(\t)\coloneqq E^\llambda[e^{iZ\cdot \t}]$ for the characteristic function of an $\R^d$-valued random variable $(Z,P^\llambda)$. To deal with the aperiodicity of $(X_n)_{n\in\N}$, we first consider the theorem for even times. Let $Y_i\coloneqq X_{2i}/2$ and note that the increments $(Y_i-Y_{i-1})_{i\geq 0}$ are $\Z^d$-valued random variable with a lattice distribution of span $1$. By \cite[Proposition 2.2.2]{LL10}, for any $x\in\Z^d$,
	\begin{align*}
		P^{\llambda}(Y_1=x)&=\frac{1}{(2\pi)^d}\int_{[-\pi,\pi]^d}\psi^{\llambda}_{Y_1}(\t)e^{-ix\t}\dd\t.
	\end{align*}
For any $\delta\in(0,\pi)$, we can now write
	\begin{align*}
		P^\llambda\big(Y_n=n\m(\llambda)+x\big)&=\frac{1}{(2\pi)^d}\int_{[-\pi,\pi]^d}\psi^{\llambda}_{Y_n}(\t)e^{-ix\cdot\t-in\m(\llambda)\cdot\t}\dd\t\\
&=I_1(x)+\frac{1}{(2\pi)^d}\int_{[-\delta,\delta]^d}\psi^{\llambda}_{Y_n}(\t)e^{-ix\cdot\t-in\m(\llambda)\cdot\t}\dd\t\\
						       &=I_1(x)+\frac{1}{(2\pi)^d}\int_{[-\delta,\delta]^d}\psi^{\llambda}_{Y_1-\m(\llambda)}(\t)^ne^{-ix\cdot\t}\dd\t\\
							&=I_1(x)+\frac{1}{(2\pi n^{1/2})^d}\int_{[-\delta n^{1/2},\delta n^{1/2}]^d}\psi^{\llambda}_{Y_1-\m(\llambda)}(\t/n^{1/2})^ne^{-ix\cdot\t/n^{1/2}}\dd\t\\
							&=I_1(x)+I_2(x)+I_3(x),
	\end{align*}
	where $I_2(x)$ and $I_3(x)$ are defined by changing the domain of integration to \\$[-\delta n^{1/2},\delta n^{1/2}]^d\setminus[-n^{1/8},n^{1/8}]^d$ and $[-n^{1/8},n^{1/8}]^d$.

	\smallskip By \cite[Lemma 2.3.2]{LL10}, it holds that $\psi^{\llambda}_{Y_1}(\t)<1$ for all $\t\in[-\pi,\pi]^d\setminus\{0\}$ (note that the lemma assumes that $Y_1$ is centered, but this is not used in the proof). Since $(\t,\llambda)\mapsto\varphi^\llambda_{Y_1}(\t)$ is jointly continuous, there exists $\eps\in(0,1)$ such that, for all $\llambda\in[-1,1]^d$ and $\t\in[-\pi,\pi]^d\setminus[-\delta,\delta]^d$
	\begin{align*}
		|\psi^\llambda_{Y_1}(\t)|\leq 1-\eps
	\end{align*}
and therefore $|I_1(x)|\leq C(1-\eps)^n$.

\smallskip Next, we show that $|I_2(x)|$ decays at a stretched exponential rate. To this end, let $\Sigma_\llambda\coloneqq \E[(Y_1-\m(\llambda))(Y_1-\m(\llambda))^T]$ be the covariance matrix. Note that since $\Sigma_\llambda$ is positive definite, all eigenvalues are positive and they depend continuously on $\llambda$. Hence there exists $r>1$ such that, for all $\t\in\R^d$ and $\llambda\in[-1,1]^d$,
	\begin{align}\label{eq:sigma}
		r^{-1}|\t|\leq t^T\Sigma_\llambda t\leq r|\t|^2.
	\end{align}
Furthermore, by applying \cite[Proposition 2.2.1(d)]{LL10} with $m=3$, there exists $C>0$ such that, for all $\t\in\R^d$ and $\llambda\in[-1,1]^d$,
\begin{align}\label{eq:taylor}
		\Big|\underbrace{\psi^\llambda_{Y_1-\m(\llambda)}(\t)-1+\frac{1}{2}\t^T\Sigma_\llambda \t}_{\eqqcolon e^\llambda(\t)}\Big|\leq R|\t|^3.
	\end{align}
Thus, if we choose $\delta\leq \frac{1}{4rR}$, then for all $|\t|\leq \delta n^{1/2}$ and $\llambda\in[-1,1]^d$,
\begin{align*}
	|\psi_{Y_1-\m(\llambda)}(\t/\sqrt n)|&=\Big|1-\frac{1}{2n}\t^T\Sigma_\llambda \t+e^{\llambda}(\t/\sqrt n)\Big|\\
					       &\leq 1-\frac{1}{2rn}|\t|^2+\frac R{n^{3/2}}{|\t|^3}\\
					       &\leq 1-\frac{r}{4rn}|\t|^2.
\end{align*}
Using the inequality $1-z\leq e^{-z}$, we therefore get
\begin{align*}
	|I_2(x)|&\leq \frac{1}{(2\pi n^{1/2})^d}\int_{[-\delta n^{1/2},\delta n^{1/2}]^d\setminus [-n^{1/8},n^{1/8}]^d}|\psi^{\llambda}_{Y_1-\m(\llambda)}(\t/\sqrt n)|^n\dd \t\\
	     &\leq\frac{1}{(2\pi n^{1/2})^d}\int_{[-\delta n^{1/2},\delta n^{1/2}]^d\setminus [-n^{1/8},n^{1/8}]^d}e^{-\frac{1}{4r}|\t|}\dd \t\\
	     &\leq Ce^{-\frac{1}{4r} n^{1/8}}.
\end{align*}
To conclude, we now introduce
\begin{align*}
	\widetilde I_3(x)\coloneqq \frac{1}{(2\pi n^{1/2})^d}\int_{[-n^{1/8},n^{1/8}]^d}e^{-\frac{1}{2}\t^T\Sigma_\llambda \t-ix\cdot\t/n^{1/2}}\dd \t
\end{align*}
and show that there exists $C>0$ such that, for all $\llambda\in[-1,1]^d$ and $x\in[-Mn^{1/2},Mn^{1/2}]^d\cap\Z^d$,
\begin{align}
	|I_3(x)-\widetilde I_3(x)|&\leq Cn^{-d/2-1/2},\label{eq:s1}\\
	C^{-1}n^{-d/2}\leq I_3(x) &\leq Cn^{-d/2}\label{eq:s2}.
\end{align}
Since we have proved that $|I_1(x)|$ and $|I_2(x)|$ are much smaller than $n^{-d/2}$ and since the error in \eqref{eq:s1} is also smaller than $n^{-d/2}$, we thus conclude that $P^\llambda(Y_n=x)$ is of order $n^{-d/2}$, for all $\llambda\in[-1,1]^d$ and $x\in[-Mn^{1/2},Mn^{1/2}]^d\cap\Z^d$.

\smallskip To prove \eqref{eq:s2}, we note that, using \eqref{eq:sigma}, we can change the domain of integration in the definition of $\widetilde{I}_3(x)$ from $[-n^{1/8},n^{1/8}]^d$ to $\R^d$ with an error term that decays at a stretched exponential rate. On the other hand, by \cite[display (2.2)]{LL10}, the resulting expression is equal to
\begin{align*}
	\frac{1}{(2\pi n^{1/2})^d\sqrt{\det\Sigma_\llambda}}e^{-\frac{1}{2n}x^T\Sigma_\llambda x}.
\end{align*}
Using again \eqref{eq:sigma}, it is not difficult to see that this expression is of order $n^{-d/2}$, uniformly in $[-Mn^{1/2},Mn^{1/2}]^d$ and $\llambda\in[-1,1]^d$.

\smallskip To prove \eqref{eq:s1}, we use \cite[display (12.3)]{LL10}, \eqref{eq:sigma} and \eqref{eq:taylor} to see that, for all $|\t|\leq n^{1/8}$ and $\llambda\in[-1,1]^d$,
\begin{align*}
	\psi^\llambda_{Y_1-\m(\llambda)}(\t/\sqrt n)^n&=\Big(1-\frac{1}{2n}\t^T\Sigma_\llambda\t+e^{\llambda}(\t/\sqrt n)\Big)^n\\
							 &=e^{-\frac{1}{2}\t^T\Sigma_\llambda \t+ne^{\llambda}(\t/\sqrt n)+O(|\t|^2/n+ne^{\llambda}(\t/\sqrt n)^2)}\\
							 &=e^{-\frac{1}{2}\t^T\Sigma_\llambda \t}\big(1+O(|\t|^3/n^{1/2}+|\t|^2/n+|\t|^6/n^{2})\big),
\end{align*}
Thus
\begin{align*}
	|I_3(x)-\widetilde I_3(x)|&\leq \frac{C}{(2\pi n^{1/2})^d}\int_{[-n^{1/8},n^{1/8}]^d}e^{-\frac{1}{2}\t^T\Sigma_\llambda \t}(|\t|^2/n+|\t|^3/n^{1/2}+|\t|^6/n^{2})\dd \t\\
				  &\leq Cn^{-d/2-1/2}\int_{\R^d}e^{-\frac{1}{2}\t^T\Sigma_\llambda \t}(|\t|^2+|\t|^3+|\t|^6)\dd \t.
\end{align*}
Using again \eqref{eq:sigma}, the last integral can be bounded independently of $\llambda\in[-1,1]^d$.

\smallskip We have shown that $P^\llambda(Y_n=x)$ is of order $n^{-d/2}$, uniformly in $\llambda\in[-1,1]^d$ and $x\in[-Mn^{1/2},Mn^{1/2}]^d\cap\Z^d$, and it only remains to derive the conclusion in term of the original random walk $X$. For $n$ even, \eqref{eq:local_max} and \eqref{eq:LCLT_lower} immediately follow. For odd $n=2k+1$, we observe
\begin{align*}
	P^\llambda(X_{2k}=x+e_1)\P^{\llambda}(X_1=e_1)\leq P^\llambda(X_{2k+1}=x)\leq \max_{|y|_1=1} P^\llambda(X_{2k}=x+y),
\end{align*}
so the conclusion follows since $P^\llambda(X_1=y)$ is bounded away from zero for $\llambda\in[-1,1]^d$ and $|y|_1=1$.
\end{proof}

\section*{Acknowledgement}
We are grateful to Ryoki Fukushima for many interesting discussions in the course of this research and to Shuta Nakajima for helpful comments on an earlier version of this manuscript. We are also grateful to two anonymous referees for their detailed and valuable feedback to this manuscript.

\bibliographystyle{plain}
\bibliography{ref.bib}

\end{document}